\documentclass[11pt,reqno,a4paper]{amsart}
\usepackage{graphicx}
\usepackage{amssymb,amsmath}
\usepackage{amsthm}
\usepackage{color,graphicx}
\usepackage{hyperref}
\usepackage{color}
\usepackage{mathabx}
\usepackage{subfigure}
\usepackage{appendix}

\usepackage{verbatim}

\usepackage{mathrsfs}

\usepackage{relsize} 

\usepackage{mathtools} 

\newcommand{\be}{\begin{equation}}

\newcommand{\ee}{\end{equation}}

\newtheorem{prop}{Proposition}[section]

\newcommand{\R}{{\mathbb R}}

\newcommand{\Ker}{{\rm \,Ker}}

\numberwithin{equation}{section}
\numberwithin{figure}{section}

\newtheorem{theorem}{Theorem}[section]
\newtheorem{proposition}[theorem]{Proposition}
\newtheorem{remark}[theorem]{Remark}
\newtheorem{lemma}[theorem]{Lemma}
\newtheorem{corollary}[theorem]{Corollary}
\newtheorem{definition}[theorem]{Definition}

\begin{document}
\vglue-1cm \hskip1cm
\title[Periodic Cubic NLS]{Orbital stability of periodic standing waves for the cubic fractional nonlinear Schr\"odinger equation}

\begin{center}



\author[G.E.B. Moraes]{Gabriel E. Bittencourt Moraes}

\address{Gabriel E. Bittencourt Moraes - State University of
	Maring\'a, Maring\'a, PR, Brazil.}
\email{pg54546@uem.br}

\author[H. Borluk]{Handan Borluk}

\address{Handan Borluk - Ozyegin University, Department of Natural and Mathematical Sciences, Cekmekoy, Istanbul, Turkey. }
\email{handan.borluk@ozyegin.edu.tr }

\author[G. de Loreno]{Guilherme de Loreno}

\address{Guilherme de Loreno - State University of
	Maring\'a, Maring\'a, PR, Brazil.}
\email{pg54136@uem.br}

\author[G.M. Muslu]{Gulcin M. Muslu}

\address{Gulcin M. Muslu - Istanbul Technical University, Department of Mathematics, Maslak,
	Istanbul,  Turkey.}
\email{gulcin@itu.edu.tr }

\author[F. Natali]{F\'abio Natali}

\address{F\'abio Natali - Department of Mathematics, State University of
	Maring\'a, Maring\'a, PR, Brazil. }
\email{fmanatali@uem.br }

\maketitle

\vspace{3mm}

\end{center}


\begin{abstract}
\noindent In this paper,  the existence and orbital stability of the periodic standing waves solutions for the nonlinear fractional Schrödinger (fNLS) equation with cubic nonlinearity is studied. The existence is determined by using a minimizing constrained problem in the complex setting and we it is showed that the corresponding real solution is always positive. The orbital stability is proved by combining some tools regarding positive operators, the oscillation theorem for fractional Hill operators and a Vakhitov-Kolokolov condition, well known for Schr\"odinger equations.
We then perform a numerical approach to generate periodic standing wave solutions of the fNLS equation by using the Petviashvili’s iteration method. We also investigate the Vakhitov-Kolokolov condition numerically which cannot be obtained analytically for some values of the order of the fractional derivative.

\end{abstract}

\section{Introduction}

In this paper, we present results concerning existence and orbital stability of periodic standing waves for the fractional nonlinear Schrödinger equation (fNLS) in the focusing case given as
\begin{equation}\label{fNLS1}
	iu_t - (-\Delta)^su + |u|^2u = 0.
\end{equation}
Here, $u:\mathbb{T}\times \mathbb{R}\longrightarrow\mathbb{C}$ is a complex-valued function and $2\pi$-periodic with respect to the first variable with $\mathbb{T}:= [-\pi, \pi]$. The fractional Laplacian $(-\Delta)^s$ is defined as a pseudo-differential
operator
\begin{equation}\label{FLaplacian}
\widehat{(-\Delta)^sg}(\xi)=|\xi|^{2s}\widehat{g}(\xi),
\end{equation}
where $\xi \in \mathbb{Z}$ and $s \in (0,1]$ (see \cite{RoncalStinga}). The fNLS equation was introduced by Laskin in \cite{Laskin2000} and \cite{Laskin2002} and it appears in several physical applications such as fluid dynamics, quantum mechanics, in the description of Boson stars and water wave dynamics (\cite{IonescuPusateri}, \cite{KirkpatrickLenzmannStaffilani} and \cite{Rabinowitz}).

Equation \eqref{fNLS1} admits the conserved quantities $E,F:H^s_{per} \rightarrow \mathbb{R}$ which are given as
\begin{equation}\label{E}
	E(u)=\frac{1}{2}\int_{-\pi}^{\pi} |(-\Delta)^{\frac{s}{2}} u|^2-\frac{1}{2}|u|^4\; dx,
\end{equation}
and
\begin{equation}\label{F}
	F(u)=\frac{1}{2} \int_{-\pi}^{\pi}|u|^2\; dx.
\end{equation}

When $s=1$, we obtain that $(-\Delta)^s=-\Delta$  is the well known Laplacian operator and \eqref{fNLS1} reduces to the cubic nonlinear Schr\"odinger equation (NLS) in the focusing case. As far as we know, there exist many applications for this specific equation such as  optics, quantum mechanics, Bose-Einstein condensates, laser beam propagation and DNA modelling. In mathematical point of view, the NLS equation describes nonlinear waves and dispersive wave phenomena (\cite{Boyd}, \cite{Cazenave}, \cite{Fibich} and \cite{SulemSulem}). In addition, there are many qualitative aspects concerning this equation and one of them is the orbital stability of standing/traveling solitary waves in one or higher dimensions. 
We refer the reader to \cite{CazenaveLions}, \cite{GrillakisShatahStraussI}, \cite{GrillakisShatahStraussII}, \cite{MaAblowitz}, \cite{Rowlands}, and \cite{WeinsteinNLS} for detailed discussion.\\
\indent A standing periodic wave solution for the equation \eqref{fNLS1} has the form
\begin{equation}\label{standingwave}
u(x,t)=e^{i\omega t}\varphi(x),
\end{equation}
where $\varphi: \mathbb{T} \longrightarrow \mathbb{R}$ is a smooth $2\pi$-periodic function and $\omega \in \mathbb{R}$ represents the wave frequency which is assumed to be positive. Substituting \eqref{standingwave} into \eqref{fNLS1}, we obtain the following differential equation with fractional derivative
\begin{equation}\label{EDO1}
(-\Delta)^s\varphi+\omega \varphi-\varphi^3=0.
\end{equation}

For $\omega>0$, let us consider the standard Lyapunov functional defined as
\begin{equation}\label{G}
G(u):=E(u)+\omega F(u).
\end{equation}
By \eqref{EDO1}, we obtain $G'(\varphi,0)=0$, that is, $(\varphi,0)$ is a critical point of $G$. In addition, the linearized operator around the pair $(\varphi,0)$ is given by
\begin{equation}\label{matrixop}
 \mathcal{L}:= G''(\varphi,0)=\begin{pmatrix}
\mathcal{L}_1 & 0 \\
0 & \mathcal{L}_2
\end{pmatrix},
\end{equation}
where
\begin{equation}\label{L1L2}
\mathcal{L}_1=(-\Delta)^s+\omega-3\varphi^2
\qquad \text{and} \qquad
\mathcal{L}_2=(-\Delta)^s+\omega-\varphi^2.
\end{equation}
Both operators $\mathcal{L}_1$ and $\mathcal{L}_2$ are self-adjoint  and they are deﬁned in $L^2_{per}$ with dense domain $H^{2s}_{per}$. Operator $\mathcal{L}$ in $(\ref{matrixop})$ plays an important role in our study.

\indent For the case $s=1$, we have the pioneer work of Angulo \cite{Angulo} where the author established results of orbital stability for positive and periodic standing waves with dnoidal profile. For this aim, the author combined the classical Floquet theory for the Hill operators $\mathcal{L}_1$ and $\mathcal{L}_2$ in $(\ref{L1L2})$ with the stability approaches in \cite{GrillakisShatahStraussI}  and \cite{WeinsteinNLS}. In the  interesting work of Gustafson \textit{et al.} in \cite{GustafsonLecozTsai}, the authors obtained cnoidal periodic wave solutions using a variational method to prove spectral
stability results with respect to perturbations with the same period $L$
and orbital stability results in the space constituted by anti-periodic functions with period $L/2$. Deconinck and Upsal in \cite{DeconinckUpsal} used the integrability of the NLS equation to determine orbital stability results for the dnoidal waves with respect to subharmonic perturbations in the space of continuous bounded functions. Additional references concerning orbital/spectral stability of periodic waves can be found in \cite{BottmanNivalaDeconinck}, \cite{ChenWenHuang}, \cite{GallayPelinovsky}, \cite{GallayHaragusI}, \cite{GallayHaragusII}, \cite{LeismanBronskiJohnsonMarangell} and \cite{natali-moraes-loreno-pastor}.

When $s \in (0,1)$, the orbital stability of real-valued, even and anti-periodic  standing wave solutions $\psi$ of \eqref{fNLS1} has been studied by Claassen and Johnson in \cite{ClaassenJohnson}. The authors determined the existence of real solutions via a minimization problem in the context of anti-periodic functions (denoted by $L^2_a(0,L)$) and they established that the associated linearized operator acting in $L^2_a(0,L)$ is non-degenerate. By assuming the additional assumption $\tfrac{d}{d\omega} \int_{0}^{L} \psi^2dx>0$ (the well-known Vakhitov-Kolokolov condition),  the authors are enabled to show that $\psi$ is orbitally stable with respect to anti-periodic perturbations in a suitable subspace of $H^s(0,L) \cap L^2_a(0,L) $.

Hakkaev and Stefanov in \cite{HakkaevStefanov} have determined the existence and the orbital (spectral)  stability of positive and periodic single-lobe solutions $\phi$ for the quadratic fractional Schr\"odinger equation
\begin{equation}\label{HakkaevStefanovEq1}
	iu_t-(-\Delta)^s u + |u|u=0 ,
\end{equation}
where $s \in \left( \tfrac{1}{4}, 1 \right)$. For the existence of periodic minimizers and stability, the authors used a (real) minimization problem as
\begin{equation}\label{HakkaevStefanovMininimization}
	\inf \left\{ \mathscr{E}(v):= \frac{1}{2} \int_{-1}^{1} ((-\Delta)^{ \frac{s}{2}}v )^2 \; dx- \frac{1}{3}\int_{-1}^{1} v^3 \; dx \; ; \; v \in H_{per}^s([-1,1]), \: \int_{-1}^{1} v^2 \;dx =\lambda \right\},
\end{equation}
where $\lambda>0$ is given. It is important to note that if a minimization problem as in $(\ref{HakkaevStefanovMininimization})$ is solved, the spectral stability of periodic waves can be established. According to \cite{GrillakisShatahStraussI}, \cite{GrillakisShatahStraussII}, \cite{natalipastor} and \cite{WeinsteinNLS}  it is necessary to determine that: \\ i) ${\rm n}(\mathcal{L})=1$ and $\Ker(\mathcal{L})=[(\phi',0),(0,\phi)]$, where ${\rm n}(\mathcal{L})$ stands the number of negative eigenvalues of $\mathcal{L}$, \\ ii) $\frac{d}{d\omega}\int_{-1}^{1}\phi^2dx>0$, \\ for the orbital stability. The first condition has been proved by the authors using that the solution $\phi$ which solves the minimizing problem $(\ref{HakkaevStefanovEq1})$ is positive (since it satisfies the equation \linebreak $\phi^2=((-\Delta)^s+\omega)\phi$, where $\omega>0$) and an oscillation theorem which is determined in \cite{ClaassenJohnson}.

Our aim in this work is to show that the standing wave solution in \eqref{standingwave}, where $\varphi=\varphi_\omega$ is a positive and single-lobe periodic wave (see Definition $\ref{single-lobe}$), is orbitally stable/unstable. According to the  sufficient conditions for the orbital stability in the energy space $H_{per}^s$  in \cite{GrillakisShatahStraussI}, we need to analyse the local and global well-posedness of the associated Cauchy problem for the fNLS equation \eqref{fNLS1}. For this important topic, we first refer to the study \cite{boling} by Boling, Yongqian and Jie. They have used Galerkin's method to give the global well-posedness results for the $n$-dimensional Cauchy problem
\begin{equation}\label{CauchyProblem1}
	\begin{cases}
		iu_t+(-\Delta)^{s} u+ \beta |u|^\rho u =0,\\
		u(x,0)=u_0(x).
	\end{cases}
\end{equation}
For $s>\tfrac{n}{2}$, global solutions in $H_{per}^s(\mathbb{T}^n)$ were established when $\beta>0$ and $\rho>0$. If $0<s< \tfrac{n}{2}$, it is necessary to assume $\rho \in \big(0, \tfrac{4s}{n-2s} \big)$ to obtain the same result. For the case  $\beta<0$, the condition for the existence of global solutions is $\rho \in \left(0, \tfrac{4s}{n}\right)$. 
Demirbas, Erdo\u gan and Tzirakis in \cite{DemirbasErdoganTzirakis} have studied the existence and uniqueness for the Cauchy problem \eqref{CauchyProblem1} for the case $n=1$, $\rho = 2$ and $\beta=1$. Using Gagliardo-Nirenberg inequality and the tools of Bourgain spaces and Strichartz estimates, the authors determined the existence of local solutions in $H^{\alpha}_{per}(\mathbb{T})$  for $\alpha> \tfrac{1-s}{2}$ and global solutions for $\alpha> \tfrac{10 s+1}{12}$.  Cho, Hwang, Kwon and Lee in \cite{ChoHwangKwonLee} used Bourgain spaces to establish local solutions in $H_{per}^{\alpha}(\mathbb{T})$ for $\alpha\geq \frac{1-s}{2}$. A refined result concerning the local well-posedness for the case $\beta=-1$ is given in \cite{Thirouin}.


Most of the authors (some of them mentioned above) apply the  Gagliardo-Nirenberg inequality for the periodic case in order to show the existence of global solutions. Nevertheless, they use the well known version posed in unbounded domains, namely,
\begin{equation}\label{GNreal}
	\|f\|_{{L^4}}^4   \leq    C  \|(-\Delta)^{\tfrac{s}{2}}f\|_{L^2}^{\tfrac{1}{s}} \, \, \|f\|_{{L^2}}^{4-\tfrac{1}{s}},
\end{equation}
where $f \in H^s$ and $C>0$ is a constant not depending on $f$.  To the  best of our knowledge, an additional term containing the $L^2$-norm needs to be added to $(\ref{GNreal})$ since it is  deduced  from the well-known inequality posed in bounded domains  (see \cite{Nirenberg}). It is important to note that the additional term containing the $L^2$-norm does not intervene in the analysis of existence of global solutions for the Cauchy problem associated to equation $(\ref{fNLS1})$ since the $L^2$-norm is a conserved quantity. Besides the orbital stability/instability results, our intention is to present a precise statement concerning the Gagliardo-Nirenberg inequality in the periodic context given by
\begin{equation}\label{GNper}
	\|f\|_{{L^4_{per}(\mathbb{T})}}^4   \leq    C  \|(-\Delta)^{\tfrac{s}{2}}f\|_{L^2_{per}(\mathbb{T})}^{\tfrac{1}{s}} \, \, \|f\|_{{L^2_{per}(\mathbb{T})}}^{4-\tfrac{1}{s}} + C \|f\|_{{L^2_{per}(\mathbb{T})}}^{4}.
\end{equation}

We now give the main points of our paper:  First, we show the existence of an even periodic single-lobe solution $\varphi$ for the equation \eqref{EDO1}.  Let $\tau>0$ be fixed. Following similar arguments as in \cite{NataliLePelinovsky} and \cite{NataliLePelinovskyKDV}, we need to solve the following constrained minimization problem
\begin{equation}\label{minimization1}
\inf\left\{\mathcal{B}_\omega(u):= \int_{-\pi}^{\pi}|(-\Delta)^{\frac{s}{2}} u|^2+\omega|u|^2\; dx \; ; \; u \in H^s_{per} ,\:  \int_{-\pi}^{\pi} |u|^4 \; dx=\tau \right\},
\end{equation}
where $\omega>0$ and $s \in \left(\tfrac{1}{4},1\right]$. Different from the approaches \cite{HakkaevStefanov}, \cite{NataliLePelinovsky} and \cite{NataliLePelinovskyKDV}, wee see that $u$ in $(\ref{minimization1})$ is complex, so that the eventual solution $\Phi$ for the mentioned problem is a complex-valued function. However, since its complex conjugate $\overline{\Phi}$ also solves the minimization problem $(\ref{minimization1})$, we can assume $\Phi=\overline{\Phi}$ in order to obtain a real valued solution $\varphi$ which is a real even single-lobe solution for the minimization problem $(\ref{minimization1})$ for all $\omega>\frac{1}{2}$.\\
\indent Another way to construct periodic real valued solutions for the equation \eqref{EDO1} can be determined by using the local and global bifurcation theory in \cite{buffoni-toland}. First, we construct small amplitude periodic solutions in the same way as in \cite{NataliLePelinovsky} (see also \cite{bruell-dhara})
 for $\omega>\frac{1}{2}$ and close to the bifurcation point $\frac{1}{2}$. After that, we give sufficient conditions to extend parameter $\omega$ to the whole interval $(\tfrac{1}{2},+\infty)$ by constructing an even periodic continuous function $\omega \in \left( \tfrac{1}{2}, +\infty \right) \longmapsto \varphi_{\omega} \in H_{per,e}^{2s}$
where $\varphi_{\omega}$ solves equation $(\ref{EDO1})$. However, since the periodic wave obtained by the global bifurcation theory can not have a single-lobe profile, we choose the periodic waves which arise as a minimum of the problem $(\ref{minimization1})$. The existence of small amplitude waves associated to the Schr\"odinger equation were determined in \cite{GallayHaragusI} for the equation $(\ref{CauchyProblem1})$ with $s=1$ and $\beta=\pm1$. First they show that these waves are orbitally
stable within the class of solutions which have the same period. For the case of general bounded perturbations, they prove that the small amplitude travelling waves are stable in the defocusing case and unstable in the focusing case.

The fact that the minimizer $\varphi$ of \eqref{minimization1} is a real even single-lobe solution for $(\ref{EDO1})$ gives us useful spectral properties which in turn play an important role regarding our stability approach. Using the fact that $\varphi$ minimizes the constrained problem in $(\ref{minimization1})$, we see that $\text{n}(\mathcal{L})=1$. Since $\mathcal{L}$ in $(\ref{matrixop})$ is a diagonal operator, it is possible to obtain by the fact \mbox{$(\mathcal{L}_1\varphi,\varphi)_{L_{per}^2}=-2\int_{-\pi}^{\pi}\varphi^4dx<0$} that $\text{n}(\mathcal{L}_1)=1$ and $\text{n}(\mathcal{L}_2)=0$ (see Section 2 for the precise notations of $\text{n}(\mathcal{L}_i)$, $i=1,2$). This means by the fact $\mathcal{L}_2\varphi=0$ that $0$ is the first eigenvalue for $\mathcal{L}_2$. A simple application of the standard Krein-Ruttman Theorem gives us $\varphi>0$, so that the solution is positive. Next, by using some facts concerning the theory of positive operators as in \cite{Albert},  the oscillation theorem in \cite{HurJohnsonMartin} and  $\varphi$ is a positive even single-lobe we obtain  $\text{z}(\mathcal{L}_1)=1$, that is, $\Ker(\mathcal{L}_1)=[\varphi']$. Here, the positivity of the single-lobe profile plays an important role in our spectral analysis since it avoids the additional assumption $1 \in {\rm R}(\mathcal{L}_1)$ as required in \cite{HurJohnson}, \cite{NataliLePelinovsky} and \cite{NataliLePelinovskyKDV} to obtain that $\text{z}(\mathcal{L}_1)=1$. All facts concerning the spectral analysis for the operators $\mathcal{L}_1$ and $\mathcal{L}_2$ in $(\ref{L1L2})$ enable us to conclude, since $\mathcal{L}$ in \eqref{matrixop} is a diagonal operator, that ${\rm n}(\mathcal{L})=1$ and ${\rm z}(\mathcal{L})=2$. In particular, as ${\rm z}(\mathcal{L}_1)=1$, we can use the implicit function theorem to construct a smooth curve
\begin{equation}\label{curveinto}
\omega \in  \left(\tfrac{1}{2}, +\infty\right) \longmapsto \varphi \in H^{s}_{per}
\end{equation}
of even and positive periodic waves with fixed period which solves \eqref{EDO1}.

The strategy to prove the orbital stability is based on an adaptation of the arguments in \cite{GrillakisShatahStraussI} and \cite{natalipastor} to the periodic setting.  Notice that $\text{n}(\mathcal{L})=1$ and $\text{z}(\mathcal{L})=2$ are useful to consider the standing wave solution in $(\ref{standingwave})$ containing only one symmetry (rotation), but the orbital stability can be considered with the orbit generated by the wave $\varphi$ containing two symmetries (namely, rotation and translation). To do so, we need to employ the stability result in \cite{natalipastor} and the existence of global solutions in time are  cornerstones for our analysis. Since we can obtain a global well-posedness result for the case $s\in(\tfrac{1}{2},1)$ according to the inequality $(\ref{GNper})$, the orbital stability of the wave can be established provided that  $\textsf{q}:=\frac{d}{d\omega}\int_{-\pi}^{\pi} \varphi^2 dx>0$. The stability result in \cite{GrillakisShatahStraussI}  can be also used for the orbital stability and yields $\textsf{q}>0$. However, we need to consider only one basic symmetry for the orbit and since we consider standing waves of the form $(\ref{standingwave})$, it is natural to consider the orbit generated by the wave constituted only by rotations. In the latter case, the energy space is the periodic Sobolev space $H_{per}^s$ restricted to the even functions, namely, $H_{per,e}^s$, instead of the usual energy space $H_{per}^s$. Restricted to this new space $L_{per,e}^2$, we have $\text{n}(\mathcal{L})=\text{z}(\mathcal{L})=1$ and this fact agrees well with the spectral (sufficient) conditions for the orbital stability in \cite{GrillakisShatahStraussI}.\\
 \indent Concerning the orbital instability, we can apply the \textit{instability theorem} in \cite{GrillakisShatahStraussI} and the fact that $\text{n}(\mathcal{L})=\text{z}(\mathcal{L})=1$ over the space $L_{per,e}^2$. Note that the orbital instability in the space $H_{per,e}^s$ will be considered in the orbit generated again by a single symmetry. Even though we are considering a smaller subspace, the orbital instability can be considered in the whole energy space $H_{per}^s$ and the orbit generated by the two symmetries. To do so, the only requirement is that $\textsf{q}<0$.\\
The above results yield the main theorem:

\begin{theorem}\label{mainth} Let $\varphi = \varphi_\omega$ be the positive and periodic single-lobe solution for the equation $(\ref{EDO1})$ obtained in Theorem $\ref{theorem1even}$, for all $\omega \in (\frac{1}{2}, +\infty)$. If $\mathsf{q}>0$, the periodic wave is orbitally stable. If $\mathsf{q}<0$, the periodic wave is orbitally unstable.
\end{theorem}
 
 \indent To obtain the sign of the quantity $\textsf{q}$ we use a numerical approach.  For this aim, we first use the Petviashvili’s iteration method to generate the periodic  standing wave solutions of the fNLS equation. Then, we use the forward difference method for the numerical differentiation with respect to $\omega$ after performing the numerical integration.

Our paper is organized as follows: In Section \ref{section2} we present some basic notations. In Section \ref{wellposedness}, we show the Gagliardo-Nirenberg inequality for fractional operators in the periodic context. The existence of even periodic minimizers with a single-lobe profile as well as the existence of small amplitude periodic waves are determined in Section \ref{evensolutions}. In Section \ref{spectralanalysis-section}, we present spectral properties for the linearized operator related to the fNLS equation and some results concerning the uniqueness of minimizers. Finally, our result about orbital stability and instability associated to periodic waves is  shown in Section \ref{stability-section}.

\section{Notation}\label{section2}

\indent For $s\geq0$, the real/complex Sobolev space
$H^s_{per}:=H^s_{per}(\mathbb{T})$
consists of all periodic distributions $f$ such that
\begin{equation}\label{norm1}
\|f\|^2_{H^s_{per}}:= 2\pi \sum_{k=-\infty}^{\infty}(1+k^2)^s|\hat{f}(k)|^2 <\infty,
\end{equation}
where $\hat{f}$ is the periodic Fourier transform of $f$ and $\mathbb{T}=[-\pi,\pi]$. The space $H^s_{per}$ is a  Hilbert space with the inner product denoted by $(\cdot, \cdot)_{H_{per}^s}$. When $s=0$, the space $H^s_{per}$ is isometrically isomorphic to the space $L^2_{per}:=H^0_{per}$ (see, e.g., \cite{IorioIorio}). The norm and inner product in $L^2_{per}$ will be denoted by $\|\cdot \|_{L_{per}^2}$ and $(\cdot, \cdot)_{L_{per}^2}$, respectively. We omit the interval $[-\pi, \pi]$ of the space $H^s_{per}(\mathbb{T})$ and we denote it by $H^s_{per}$ shortly. In addition, the norm  in \eqref{norm1} can be written as (see \cite{Ambrosio})
\begin{equation}\label{norm}
	\|f\|_{ H^{s}_{per}}^2=\|(-\Delta)^{\tfrac{s}{2}}f\|_{L^2_{per}}^2+\|f\|_{L^2_{per}}^2.
\end{equation}

For $s\geq0$, the space
$
H^s_{per,e}:=\{ f \in H^s_{per} \; ; \; f \:\; \text{is an even function}\} 
$
is endowed with the same norm and inner product in $H^s_{per}$. If it is needed, the  above notations  can  be extended in the complex/vectorial case in the following sense:  $f\in {H}_{per}^s\times H_{per}^s$ we have $f=f_1+if_2\equiv (f_1,f_2)$, where $f_i\in H_{per}^s$  $(i=1,2)$  since $\mathbb{C}$ is identified with $\mathbb{R}^2$.


We denote the number of negative eigenvalues and the dimension of the kernel of a certain linear operator $\mathcal{A}$, by $\text{n}(\mathcal{A})$ and $\text{z}(\mathcal{A})$, respectively.

\section{Gagliardo-Nirenberg inequality in the fractional periodic context}\label{wellposedness}

In this section, we show the Gagliardo-Nirenberg inequality for fractional operators in the periodic case. Our intention is to give a precise result of global well-posedness associated to the following Cauchy problem
\begin{equation}\label{OurCauchyProblem1}
\begin{cases}
iu_t-(-\Delta)^{s} u+u|u|^{2}=0, \\
u(x,0)=u_0(x).
\end{cases}
\end{equation}
\indent For this aim,  we need  the Gagliardo-Nirenberg inequality for bounded domains of cone-type $\Omega \subset \mathbb{R}^n$, $n \in \mathbb{N}$ (for details of this kind of domains, see \cite[Section 4.2.3, Equation 7]{Triebel1978}) stated in the next lemma. In the rest of this section,  we consider the fractional Sobolev space $H^r_{q}(\Omega)=W^{r,q}(\Omega)$, well known as \textit{Slobodeckij space} (for details, see \cite[Section 1.2]{bisci-radulescu-servadei}, \cite[Section 2.3.3, Equation 1]{Triebel1978} and \cite[Section 4.2.1, Definition 1]{Triebel1978}) for each $r\in [0,1)$ and $q \geq 1$. In what follows, we  handle with real-valued functions. For complex-valued functions, the arguments are similar.

\begin{lemma}[Gagliardo-Nirenberg inequality for bounded domains of cone-type]\label{LemmaGPN}
Let  $\Omega \subset \mathbb{R}^n$ be a bounded domain of cone-type. If $k,s \in (0,1)$, $p,p_0>1$ and $r>0$ satisfy 
$$
r=k s \quad \text{and} \quad \frac{1}{p}=\frac{1-k}{p_0}+\frac{k}{2},
$$
then there exists $C_1>0$ such that,
\begin{equation}\label{T.00}
\|f\|_{H^r_p(\Omega)} \leq C_1 \|f\|^{1-k}_{L^{p_0}(\Omega)} \, \|f\|^{k}_{H^s(\Omega)},
\end{equation}
for all $ f \in L^{p_0}(\Omega) \cap H^{s}(\Omega)$.
\end{lemma}
\begin{proof}
First of all,  according to \cite[Section 4.3.1, Theorem 2]{Triebel1978} the relation of interpolation
$$
(H^{s_0}_{p_0}(\Omega), H^{s}_{p_1}(\Omega))_{k}=H^r_p(\Omega),
$$
is valid. Here, $p,p_1,p_0>1$, $k \in (0,1)$, $s_0,s \geq 0$, and $r>0$ satisfy
$$
r=s_0 (1-k)+ k s \quad \text{and} \quad \frac{1}{p}=\frac{1-k}{p_0}+\frac{k}{p_1}.
$$

As a consequence of \cite[Section 1.3.3, Equation 5]{Triebel1978} there exists a constant a ${C_0}>0$ such that
\begin{equation}\label{T.01}
\|f\|_{H^{r}_{p}(\Omega)} \leq {C_0}  \|f\|_{H^{s_0}_{p_0}(\Omega)}^{1-k}\|f\|_{H^{s}_{p_1}(\Omega)}^{k},
\end{equation}
for all $f \in H^{s_0}_{p_0}(\Omega) \cap H^{s}_{p_1}(\Omega).$
In particular, by considering $p_1=2$, $s_0=0$ and $s \in (0,1)$, we see that
$$
r= k s \in (0,1), \qquad  \frac{1}{p}=\frac{1-k}{p_0}+\frac{k}{2}.
$$
\indent Thus, by \eqref{T.01} we obtain
\begin{equation*}
\|f\|_{H^r_p(\Omega)} \leq {C_1} \|f\|^{1-k}_{L^{p_0}(\Omega)} \, \|f\|^{k}_{H^s(\Omega)},
\end{equation*}
for some constant $C_1 > 0$ and for all $f \in L^{p_0}(\Omega) \cap H^{s}(\Omega)$.

\end{proof}

\begin{corollary}
Let  $\Omega \subset \mathbb{R}^n$ be a bounded domain of cone-type. If $k,s \in (0,1)$ and $p,p_0>1$ satisfy
\begin{equation}\label{condpq}
\frac{1}{p}=\frac{1-k}{p_0}+\frac{k}{2},
\end{equation}
there exists $C_2>0$ such that,
\begin{equation*}
\|f\|_{L^p(\Omega)} \leq C_2 \|f\|^{1-k}_{L^{p_0}(\Omega)} \, \|f\|^{k}_{H^s(\Omega)},
\end{equation*}
for all $ f \in L^{p_0}(\Omega) \cap H^{s}(\Omega)$.
\end{corollary}
\begin{proof}
First, it is clear that $r=k s \in (0,1)$. The Sobolev embedding $H^r_p(\Omega) \hookrightarrow L^p(\Omega)$, condition $(\ref{condpq})$, and Lemma \ref{LemmaGPN} give us
$$
\|f\|_{L^p(\Omega)} \leq C_2 \|f\|^{1-k}_{L^{p_0}(\Omega)} \, \|f\|^{k}_{H^s(\Omega)},
$$
for some constant $C_2>0$.
\end{proof}

As a particular case of the Lemma \ref{LemmaGPN} in the periodic context, we established the following theorem.

\begin{theorem}[$n-$dimensional periodic Gagliardo-Nirenberg inequality]\label{PGN}
Let $\mathbb{T}^n \subset \mathbb{R}^n$ be the $n$-dimensional torus. If $k,s \in (0,1)$ and $r>0$ are so that $r=k s $, then there exists ${C_3}>0$ such that
\begin{equation}\label{PGNI}
\|f\|_{H^r_{per}(\mathbb{T}^n )} \leq {C_3} \|f\|^{1-k}_{L^{2}_{per}(\mathbb{T}^n )} \, \|f\|^{k}_{H^s_{per}(\mathbb{T}^n )},
\end{equation}
for all $ f \in  H^{s}_{per}(\mathbb{T}^n)$.
\end{theorem}
\begin{proof}
Since the $n$-dimensional torus $\mathbb{T}^n\subset \mathbb{R}^n$ is a bounded domain of cone-type (\cite[Section 4.2.3, Remark 5]{Triebel1978}), we obtain that the Lemma \ref{LemmaGPN} is valid for $\Omega=\mathbb{T}^n$, $p_0=2$ and $r=k s \in (0,1)$. Moreover, by \cite[Section 4.6.1, Equation 2]{Triebel1978} and \cite[Section 9.1.3, Remark 1]{Triebel2010} the norms in $H^s_{per}(\Omega)$ and $H^s(\Omega)$ are equivalent and since $L^m(\mathbb{T}^n) \equiv L^m_{per}(\mathbb{T}^n)$, for all $m \geq 1$, it follows by \eqref{T.00} the following inequality
\begin{equation*}
\|f\|_{H^r_{per}(\mathbb{T}^n )} \leq C_3 \|f\|_{{L^{2}_{per}}(\mathbb{T}^n )}^{1-k} \, \|f\|_{H^{s}_{per}(\mathbb{T}^n )}^{k},
\end{equation*}
for all  $f \in  H^{s}_{per}(\mathbb{T}^n )$ and for some constant $C_3>0$.
\end{proof}

\begin{corollary}[$1$-dimensional Periodic Gagliardo-Nirenberg inequality]\label{PGN1}
	Let  $s  \in \left( \tfrac{1}{4}, 1\right)$ be fixed. There exists a constant $C_4>0$ such that,
	\begin{equation}\label{GN-s1}
\|f\|_{{L^4_{per}}}^4   \leq    C_4  \|(-\Delta)^{\tfrac{s}{2}}f\|_{L^2_{per}}^{\tfrac{1}{s}} \, \, \|f\|_{{L^2_{per}}}^{4-\tfrac{1}{s}} + C_4 \|f\|_{{L^2_{per}}}^{4},
	\end{equation}
for all $f \in H^s_{per}$.
\end{corollary}

\begin{proof}
In Theorem $\ref{PGN}$, let us consider $n=1$. We have
$$
\|f\|_{H^r_{per}} \leq{{C_3}} \|f\|_{{L^{2}_{per}}}^{1-k} \, \|f\|_{H^{s}_{per}}^{k},
$$
for all $f \in  H^{s}_{per}$. Here, we  consider $r=k s \in (0,1)$, $ k \in (0,1)$, and $s \in (0,1)$. \\
\indent Several calculations and  the definition of the norm of $H^s_{per}$ given by \eqref{norm} yield the existence of a constant $C_5>0$ where
\begin{eqnarray*}
\|f\|_{{H^r_{per}}}^{4}   \leq   C_5 \|f\|_{{L^2_{per}}}^{4(1-k)} \, \left(   \|(-\Delta)^{\tfrac{s}{2}}f\|_{L^2_{per}}^2+\|f\|_{L^2_{per}}^2   \right)^{2  k},
\end{eqnarray*}
for all $f \in H^{s}_{per}.$ By \cite[Lemma 3.197]{IorioIorio}, we obtain the existence of a constant $C_6>0$ such that
$$
\left(   \|(-\Delta)^{\tfrac{s}{2}}f\|_{L^2_{per}}^2+\|f\|_{L^2_{per}}^2   \right)^{2k} \leq  C_6 \left(   \|(-\Delta)^{\tfrac{s}{2}}f\|_{L^2_{per}}^{4k}+\|f\|_{L^2_{per}}^{4k}   \right).
$$
Thus, there exists a constant $C_7 > 0$ such that
\begin{eqnarray}\label{GN1}
\|f\|_{{H^r_{per}}}^4    \leq   C_7 \|f\|_{{L^2_{per}}}^{4(1-k)} \,  \|(-\Delta)^{\tfrac{s}{2}}f\|_{L^2_{per}}^{4k} + C_7 \|f\|_{{L^2_{per}}}^{4}.
\end{eqnarray}

Choosing $r=\tfrac{1}{4}$ and using the embedding $H^{r}_{per}   \hookrightarrow     L^4_{per}$ (see \cite[Theorem 4.2]{Ambrosio}), we obtain from \eqref{GN1} for $s=\tfrac{1}{4k}  \in \left( \tfrac{1}{4}, 1\right)$ that
\begin{eqnarray*}\label{GNP1}
\|f\|_{{L^4_{per}}}^4   \leq    C_4 \|(-\Delta)^{\tfrac{s}{2}}f\|_{L^2_{per}}^{\tfrac{1}{s}} \, \|f\|_{{L^2_{per}}}^{4-\tfrac{1}{s}} + C_4 \|f\|_{{L^2_{per}}}^{4},
\end{eqnarray*}
for some constant $C_4>0$.
\end{proof}

 The existence of global solutions in time for the Cauchy problem associated to the equation $(\ref{fNLS1})$ is obtained by  the combination of Corollary $\ref{PGN1}$ and the conserved quantities $E$ and $F$ given  by \eqref{E} and $\eqref{F}$, respectively. In fact, as mentioned in the Introduction, we see that for $s \in \left( \tfrac{1}{2}, 1 \right)$, there exists a local solution $u \in C([0,T], {H}^s_{per})$ of the Cauchy problem \eqref{OurCauchyProblem1} associated to the equation $(\ref{fNLS1})$ with initial data $u_0 \in{H}^s_{per}$ (see \cite{DemirbasErdoganTzirakis} and \cite{ChoHwangKwonLee}). For all $t\geq0$, we have
\begin{equation*}
\|(-\Delta)^{\frac{s}{2}}u(t)\|^2_{L^2_{per}} = 2E(u_0) + \frac{1}{2}\|u(t)\|^4_{L^4_{per}}.
\end{equation*}

By Corollary $\ref{PGN1}$ we obtain the existence of a constant $C>0$ such that
\begin{equation}\label{GN-global}
	\|(-\Delta)^{\frac{s}{2}}u(t)\|^2_{L^2_{per}} \leq 2E(u_0) +  C \|u_0\|^{4-\frac{1}{s}}_{L^2_{per}}\|(-\Delta)^{\frac{s}{2}}u(t)\|^{\frac{1}{s}}_{L^2_{per}} + C\|u_0\|^4_{L^2_{per}},
\end{equation}
where we are using the fact that the  $L^2_{per}$-norm is a conserved quantity (see \eqref{F}).

\indent Therefore, by \eqref{GN-global}, we obtain the following scenario for global solutions $H_{per}^s$:
\begin{itemize}
	\item When $s \in \left(\frac{1}{2},1\right]$, we can proceed similarly to the authors in \cite{boling} to conclude the existence of global solutions in time.
	
	\item When $s=\frac{1}{2}$, we use again \cite{boling} to conclude the existence of global solutions in time for $||u_0||_{L_{per}^2}$ small enough. It is also expected blow-up in finite time for large $||u_0||_{L_{per}^2}$.

\end{itemize}

Summarizing our analysis performed above, we obtain the following global well-posedness result for the Cauchy problem associated to the fNLS equation $(\ref{fNLS1})$.

\begin{proposition} \label{gwp}
Let $s \in \left( \frac{1}{2},1 \right]$. The Cauchy problem associated to the equation $(\ref{fNLS1})$ is globally well-posed in ${H}^{s}_{per}$. More precisely, for any $u_0 \in {H}^{s}_{per}$ there exists an unique global  solution \linebreak $u \in C([0,+\infty), {H}^s_{per})$ such that $u(0)=u_0$ and it satisfies \eqref{fNLS1}. Moreover, for each $T>0$ the mapping
$$
u_0 \in  {H}^{s}_{per} \longmapsto  u \in C([0,T], {H}^s_{per})
$$
is continuous.\end{proposition}

\section{Existence of periodic waves}\label{evensolutions}

In this section, we prove the existence of the even periodic wave solutions of \eqref{EDO1} using two approaches. First, we use a variational characterization by minimizing a suitable constrained functional to obtain positive and even periodic waves with single-lobe profile. Second, we present some tools concerning the existence of small amplitude periodic waves using bifurcation theory. In addition, it is possible to show that such waves are also solutions for the minimization problem presented in the next subsection.

\subsection{Existence of periodic waves via minimizers}

In this subsection, we prove the existence of even periodic solutions for \eqref{EDO1}  by considering the variational problem given by \eqref{minimization1}. First, we define the of solution with \textit{single-lobe} profile.
\begin{definition}\label{single-lobe}
We say that a periodic wave satisfying the  equation \eqref{EDO1} has single-lobe profile if there exist only one maximum and minimum on $[-\pi,\pi]$. Without loss of generality, we assume that the  maximum point occurs at $x=0$.
\end{definition}

For $\tau>0$, let us consider the set
\begin{equation}\label{minimizerset}
\mathcal{Y}_\tau:=\left\{u \in {H}^s_{per} \; ; \; \|u\|_{L_{per}^4}^4=\tau \right\}.
\end{equation}
 For $\omega>0$, we define the functional $\mathcal{B}_\omega :  {H}^s_{per} \longrightarrow \mathbb{R}$ given by
\begin{equation}\label{functional1}
\mathcal{B}_\omega (u):= \frac{1}{2}\int_{-\pi}^{\pi} |(-\Delta)^{\frac{s}{2}} u|^2+\omega\,|u|^2\; dx,
\end{equation}
$\text{for all} \; u \in {H}^s_{per}.$

We see that
\begin{equation}\label{positivity}
\mathcal{B}_\omega (u) \geq 0 \quad \text{and} \quad G(u) \leq \mathcal{B}_\omega (u)
\end{equation}
We have the following result of existence:

\begin{proposition}\label{theorem1complex}
Let $s \in \left(\tfrac{1}{4}, 1 \right]$ and $\tau, \omega>0$ be fixed. The minimization problem
\begin{equation}\label{minimizationproblem}
\Gamma_\omega:= \inf_{u \in \mathcal{Y}_\tau} \mathcal{B}_\omega(u)
\end{equation}
has at least one solution, that is, there exists a complex-valued function $\Phi \in\mathcal{Y}_\tau$ such that\linebreak $\mathcal{B}_\omega(\Phi)=\Gamma_\omega$. Moreover, $\Phi$ satisfies
$$
(-\Delta)^s \Phi +\omega \Phi - |\Phi|^2 \Phi =0.
$$
\end{proposition}
\begin{proof}
First we claim that the functional $\mathcal{B}_\omega$ induces an equivalent norm in $ {H}^s_{per}$. Indeed, being the norm in $ {H}^s_{per}$ given as in \eqref{norm} and since the functional $\mathcal{B}_\omega$ can be written as
$$
2\mathcal{B}_\omega(u)=\|(-\Delta)^{\frac{s}{2}} u\|_{{L}_{per}^2}^2+\omega\|u\|^2_{{L}_{per}^2},\; u \in  {H}^s_{per}.
$$
It is easy to see that there exist constants $c_0, c_1>0$ so that
\begin{equation}\label{9}
0 \leq c_0\|u\|_{{H}_{per}^s} \leq  \sqrt{2\mathcal{B}_\omega(u)} \leq c_1 \|u\|_{{H}_{per}^s}.
\end{equation}
Moreover, by \eqref{positivity}, one has $\Gamma_\omega \geq 0$. 

Using the smoothness of the functional $\mathcal{B}_\omega$, we may consider a sequence of minimizers \\ \mbox{$(u_n)_{n\in \mathbb{N}} \subset Y_\tau$} such that
\begin{equation}\label{10}
\mathcal{B}_\omega(u_n) \longrightarrow \Gamma_\omega, \; \; \; n \rightarrow \infty.
\end{equation}

By \eqref{10}, we have that the sequence $(\mathcal{B}_\omega(u_n))_{n \in \mathbb{N}} \subset \mathbb{R}$ is bounded, so that it is bounded in ${H}_{per}^s$. Since $s\in \left(\frac{1}{4},1\right]$ and the Sobolev space ${H}^s_{per}$ is reflexive, there exists $\Phi \in {H}^s_{per}$ such that (modulus a subsequence),
\begin{equation}\label{11}
u_n \xrightharpoonup{\quad \:} \Phi \; \text{weakly in} \; {H}^s_{per}.
\end{equation}

Again, since $s \in \left( \tfrac{1}{4},1\right]$,  we obtain that the embedding
\begin{eqnarray}\label{12}
{H}^s_{per} \xhookrightarrow{\quad \:}  {L}^4_{per}
\end{eqnarray}
is compact (see \cite[Theorem 2.8]{BergerSchechter} or \cite[Theorem 5.1]{Amann}). Thus, modulus a subsequence we also have
\begin{equation}\label{13}
u_n \longrightarrow \Phi \; \text{in} \; {L}^4_{per}.
\end{equation}
Moreover, using the estimate
\begin{eqnarray*}
\bigg|\int_{-\pi}^{\pi} \big(|u_n|^4-|\Phi|^4\big)\; dx \bigg| & \leq & \int_{-\pi}^{\pi} \big||u_n|^4-|\Phi^4|\big|\; dx \\
& \leq & \big( \|\Phi\|^3_{{L}_{per}^4}+\|\Phi\|_{{L}_{per}^4}^2\,\|u_n\|_{{L}_{per}^4}+\|\Phi\|_{{L}_{per}^4}\,\|u_n\|^2_{{L}_{per}^4}+\|u_n\|_{{L}_{per}^4}^3\big)\|u_n-\Phi\|_{{L}_{per}^4}
\end{eqnarray*}
and \eqref{13}, it follows that $\|\Phi\|_{L_{per}^4}^4=\tau$. 
Furthermore, since $\mathcal{B}_\omega$ is  lower semi-continuous, we have
$$
\mathcal{B}_\omega(\Phi) \leq \liminf_{n \to \infty} \mathcal{B}_\omega(u_n)
$$
that is,
\begin{equation}\label{14}
\mathcal{B}_\omega(\Phi) \leq  \Gamma_\omega.
\end{equation}

On the other hand, once $\Phi$ satisfies $\|\Phi\|_{L_{per}^4}^4=\tau$, we obtain
\begin{equation}\label{15}
\mathcal{B}_\omega(\Phi) \geq \Gamma_\omega.
\end{equation}

By \eqref{14} and \eqref{15} we conclude
$$
\mathcal{B}_\omega(\Phi)= \Gamma_\omega=\inf_{u \in \mathcal{Y}_\tau} \mathcal{B}_\omega(u).
$$
In other words, the function $\Phi \in \mathcal{Y}_\tau \subset {H}^s_{per}$ is a minimizer of the problem \eqref{minimizationproblem}. Note that since $\tau>0$, we see that $\Phi$ is a complex-valued function such that $\Phi \notequiv 0$.

Next, by the Lagrange Multiplier Theorem, there exists a constant $c_2 \in \mathbb{R}$ so that
$$
(-\Delta)^s \Phi+\omega \Phi=c_2 |\Phi|^2 \Phi.
$$
A standard scaling argument  allows us to choose $c_2=1$ (see  \cite[page 10]{AmaralBorluk}). Thus, we have that $\Phi$ is a periodic minimizer of the problem $(\ref{minimization1})$ and it satisfies the equation
\begin{equation}\label{EDO2evencomplex}
(-\Delta)^s \Phi +\omega \Phi - |\Phi|^2 \Phi =0.
\end{equation}
\end{proof}


\begin{remark}\label{remarkreal}
Let $\Phi \in {H}^s_{per}$ be the minimizer obtained by Theorem \ref{theorem1complex}. It is easy to check (see, for instance, \cite[\textit{Lemma 2.2}]{ClaassenJohnson}) that $\overline{\Phi}$ satisfies
$$
\mathcal{B}_\omega(\overline{\Phi})=\Gamma_\omega.
$$
In addition,
$$
(-\Delta)^s \overline{\Phi} +\omega \overline{\Phi} - |\overline{\Phi}|^2 \overline{\Phi} =0.
$$

In order to guarantee the existence of real-valued solutions for the equation $(\ref{EDO2evencomplex})$, we are going to assume that $\Phi=\overline{\Phi}$. Thus, $\Phi$ has the form $\Phi=\varphi+i0$, where $\varphi \in H^s_{per}$ satisfies $(\ref{EDO1})$ and the minimization problem
\begin{equation}\label{realminimizationproblem}
\mathcal{B}_{\omega}(\varphi) \equiv \mathcal{B}_{\omega}(\Phi) =  \Gamma_\omega.
\end{equation}
\end{remark}

As a consequence of the assumption in Remark \ref{remarkreal}, we have the following result.

\begin{proposition}[Existence of Even Solutions]\label{theorem1even}
Let $s \in \left(\tfrac{1}{4}, 1 \right]$ and $\omega>0$ be fixed. Let $\varphi \in H^s_{per}$ be the real-valued periodic minimizer  given by the Remark \ref{remarkreal}. If $\omega \in \left(0, \tfrac{1}{2}\right]$ then $\varphi$ is the constant solution and if $\omega \in \left(\tfrac{1}{2}, +\infty \right)$ then $\varphi$ is an even periodic single-lobe solution for the equation $(\ref{EDO1})$.
\end{proposition}
\begin{proof}
First, by a bootstrapping argument we infer that $\varphi \in H^\infty_{per}$ (see \cite[Propostion 3.1]{Cristofani} and \cite[Proposition 2.4]{NataliLePelinovsky}). In addition, the solution $\varphi$ can be assumed even\footnote{Since we can minimize the functional $\mathcal{B}_\omega$ over the space $H^s_{per,e}$ in order to obtain an even minimizer $\varphi$, as in \cite[Theorem 4.1]{NataliLePelinovskyKDV}}.

Since the solution can be constant, we need to avoid this case in order to guarantee that the minimizer has a single-lobe profile. First, we note that the positive constant solution of the equation \eqref{EDO1} is $\varphi\equiv \sqrt{\omega}$. In this case the operator $\mathcal{L}_1= (-\Delta)^s+\omega-3\varphi^2$ is written as
$
\mathcal{L}_1= (-\Delta)^s-2\omega.
$
As a result of \cite[Example 4.4]{HurJohnsonMartin}, we obtain that ${\rm n}(\mathcal{L}_1)=1$ if and only if $\omega \in \left(0, \tfrac{1}{2}\right]$. On the other hand, it is easy to see that if $\varphi$ is nonconstant then $\mathcal{L}_1(\varphi')=0$ so that ${\rm n}(\mathcal{L}_1)=1$. Thus, we conclude that the constant solution $\varphi=\sqrt{\omega}$ is a mininimizer of \eqref{minimizationproblem} for $\omega \in \left(0, \tfrac{1}{2}\right]$ and for $\omega \in \left(\tfrac{1}{2}, +\infty \right)$, solution  $\varphi$ is a nonconstant minimizer. Furthermore, in the latter case,  we can consider the symmetric rearrangements $\varphi^{\star}$ associated to $\varphi$ and it is well known that such rearrangements are invariant under the constraint of $\mathcal{Y}_\tau$ by using \cite[Appendix A]{ClaassenJohnson}. Moreover, due to the fractional Polya-Szeg\"o inequality, in \cite[Lemma A.1]{ClaassenJohnson}, 
 we have
$$
\int_{-\pi}^\pi \big((-\Delta)^{\frac{s}{2}}\varphi^{\star}\big)^2\;dx \leq \int_{-\pi}^\pi\big((-\Delta)^{\frac{s}{2}}\varphi\big)^2\;dx .
$$

Thus, by \eqref{realminimizationproblem}, we obtain $\mathcal{B}_\omega(\varphi^{\star}) = \Gamma_\omega$ with $\varphi^{\star}$ being symmetrically decreasing away from the maximum point $x=0$. To simplify the notation, we assume that $\varphi=\varphi^{\star}$, so that $\varphi$ has an even single-lobe profile according to the Definition $\ref{single-lobe}$.
\end{proof}
\subsection{Small-amplitude periodic waves}

The existence and convenient formulas for the small amplitude periodic waves associated to the equation \eqref{EDO1} will be shown in this subsection. After that, we show that the local bifurcation theory used to determine the existence of small amplitude waves can be extended and the local solutions can be considered as global for a fixed $\omega>\frac{1}{2}$. This fact is a very important feature in our context since it can be used as an alternative form to prove the existence of periodic even solutions (not necessarily having a single-lobe profile) for the equation $(\ref{EDO1})$ when $s\in (0,1)$. To do so, we use the theory contained in \cite[Chapters 8 and 9]{buffoni-toland}.

First, we shall give some steps to prove the existence of small amplitude periodic waves. In fact, for $s \in (0,1]$, let $F: H_{per,e}^{2s} \times (\tfrac{1}{2},+\infty) \rightarrow L_{per,e}^2$ be the smooth map defined by
\begin{equation}\label{F-lyapunov}
	F(g,\omega) = (-\Delta)^s g + \omega g - g^3.
\end{equation}
We see that $F(g,\omega) = 0$ if and only if $g \in H_{per,e}^{2s}$ satisfies \eqref{EDO1} with corresponding wave frequency $\omega \in (\tfrac{1}{2}, +\infty)$. The Fr\'echet derivative of the function $F$ with respect to the first variable is then given by
\begin{equation}\label{Dg}
	D_g F(g, \omega) f = \left( (-\Delta)^s + \omega - 3g^2 \right) f.
\end{equation}
Let $\omega_0 > \frac{1}{2}$ be fixed. At the point $(\sqrt{\omega_0}, \omega_0)$, we have that
\begin{equation}
	D_gF(\sqrt{\omega_0}, \omega_0) = (-\Delta)^s + \omega_0 - 3(\sqrt{\omega_0})^2 = (-\Delta)^s - 2\omega_0.
\end{equation}

The nontrivial kernel of $D_gF(\sqrt{\omega_0}, \omega_0)$ is determined by functions $h \in H_{per,e}^{2s}$ such that
\begin{equation}
	\widehat{h}(k) (-2\omega_0 + |k|^{2s}) = 0.
\end{equation}
We see that $D_g F(\sqrt{\omega_0}, \omega_0)$ has the one-dimensional kernel if and only if $\omega_0 = \frac{\:\:|k|^{2s}}{2}$ for some $k \in \mathbb{Z}$. In this case, we have
\begin{equation}
	{\rm Ker}D_gF(\sqrt{\omega_0}, \omega_0) = [\tilde{\varphi}_k],
\end{equation}
where $\tilde{\varphi}_k(x) = \cos(kx)$.

The local bifurcation theory contained in \cite[Chapter 8.4]{buffoni-toland} enables us to guarantee the existence of an open interval $I$ containing $\omega_0 > \tfrac{1}{2}$, an open ball $B(0,r) \subset H_{per,e}^{2s}$ for some $r>0$ and a unique smooth mapping
$$\omega \in I \longmapsto \varphi:= \varphi_\omega \in B(0,r) \subset H_{per,e}^{2s}$$
such that $F(\varphi, \omega) = 0$ for all $\omega \in I$ and $\varphi\in B(0,r)$.

For each $k \in \mathbb{N}$, the point $(\sqrt{\tilde{\omega}_k}, \tilde{\omega}_k)$ where $\tilde{\omega}_k := \frac{|k|^{2s}}{2}$ is a bifurcation point. Moreover, there exists $a_0 > 0$ and a local bifurcation curve
\begin{equation}\label{localcurve}
	a \in (0,a_0) \longmapsto (\varphi_{k,a}, \omega_{k,a}) \in H_{per,e}^{2s} \times (0,+\infty)
\end{equation}
which emanates from the point $(\sqrt{\tilde{\omega}_k}, \tilde{\omega}_k)$ to obtain small amplitude even $\frac{2\pi}{k}$-periodic solutions for the equation \eqref{EDO1}. In addition, we have $\omega_{k,0} = \tilde{\omega}_k$, $D_a \varphi_{k,0} = \tilde{\varphi}_k$ and all solutions of $F(g, \omega) = 0$ in a neighbourhood of $(\sqrt{\tilde{\omega}_k}, \tilde{\omega}_k)$ belongs to the curve in $(\ref{localcurve})$ depending on $a \in (0,a_0)$.



\begin{proposition}
	
	Let $s \in (0,1]$ be fixed. There exists $a_0 > 0$ such that for all $a \in (0,a_0)$ there is a unique even local periodic solution $\varphi$ for the problem \eqref{EDO1}. The small amplitude periodic waves are given by the following expansion:
	\begin{equation}\label{varphi-stokes2}
		\varphi(x) = \sqrt{\omega} + \sqrt{2} \phi(x).
	\end{equation}
	where 
	\begin{equation}\label{varphi-stokes31}
		\phi(x) = a \phi_1(x) + a^2 \phi_2(x) + a^3 \phi_3(x) + \mathcal{O}(a^4),
	\end{equation}
	Here $\phi_1(x) = \cos(x)$,
	$$\phi_2(x) = -\frac{3}{2} + \frac{3}{2(2^{2s} - 1)} \cos(2x),$$
	$$\phi_3(x) = \frac{1}{2(3^{2s} - 1)} \left[ 1 + \frac{9}{2^{2s} - 1}\right] \cos(3x),$$
	and
	$$\gamma = \frac{15}{2} - \frac{9}{2(2^{2s} - 1)}.$$
	The frequency $\omega$ in this case is expressed as	
	
	\begin{equation}\label{varphi-stokes3}
		\omega = \frac{1}{2} + a^2 \gamma + \mathcal{O}(a^4).
	\end{equation}
	For $s\in (\tfrac{1}{4},1]$, the pair $(\varphi, \omega) \in H_{per,e}^{s} \times (\frac{1}{2}, +\infty)$ is global in terms of the parameter $\omega > \frac{1}{2}$ and it satisfies \eqref{EDO1}.
\end{proposition}

\begin{proof}
	The first part of the proposition has been already determined in $(\ref{localcurve})$ by considering $k=1$. To get the expression in \eqref{varphi-stokes2}, we use  arguments similar to the ones in \cite[Section 5]{NataliLePelinovsky}. To obtain that the  local curve \eqref{localcurve} extends to a global one for the case $s\in(\tfrac{1}{4},1]$, we first need to prove  that $D_gF(g,\omega)$ given by \eqref{Dg} is a Fredholm operator of index zero. Indeed, we define the set $S = \{(g, \omega) \in D(F) : F(g, \omega) = 0\}$. Let $(g, \omega) \in H_{per,e}^{2s} \times (\tfrac{1}{2},+\infty)$ be a solution of $F(g, \omega) = 0$. For $Y:=L_{per,e}^{2}$  we have that
\begin{equation}\label{fredalt}
	\mathcal{L}_{1|_{Y}} \psi \equiv D_g F(g, \omega) \psi = \left( (-\Delta)^s - 3g^2\right) \psi + \omega\psi = 0,\end{equation}
has two linearly independent solutions and at most one belongs to $H_{per,e}^{2s}$ (see \cite[Theorem 3.12]{ClaassenJohnson}). If there are no solutions in $H_{per,e}^{2s}\backslash\{0\}$, then the problem  $\left( (-\Delta)^s + \omega - 3 g^2\right) \psi = f$ has a unique non-trivial solution $\psi \in H_{per,e}^{2s}$ for all $f \in Y$ since ${\rm Ker}(\mathcal{L}_{1|_Y})^{\bot} = R(\mathcal{L}_{1|_Y}) = Y$.

On the other hand, if there is a solution $e \in H_{per,e}^{2s}$ we obtain by standard Fredholm Alternative that $(\ref{fredalt})$ has a solution if and only if
$$\int_{-\pi}^\pi e(x) f(x) dx = 0,$$
for all  $f \in Y$. We can conclude in both cases that the Fr\'echet derivative of $F$ in terms of $g$ given by \eqref{Dg} is a Fredholm operator of index zero.

\indent Let us prove that every bounded and closed  $S$ is a compact set on $H_{per,e}^{2s}\times (\tfrac{1}{2},+\infty)$. For  $g\in H_{per,e}^{2s}$ and $\omega>\frac{1}{2}$, we define $\widetilde{F}(g,\omega)=((-\Delta)^s+\omega)^{-1}g^3$. Since $s\in(\tfrac{1}{4},1]$, we see that $\widetilde{F}$ is well defined since $H_{per,e}^{2s}$ is a Banach algebra, $(g,\omega)\in S$ if and only if $\widetilde{F}(g,\omega)=g$ and $\widetilde{F}$ maps $H_{per,e}^{2s}\times (\tfrac{1}{2},+\infty)$ into $H_{per,e}^{4s}$. The compact embedding $H_{per,e}^{4s}\hookrightarrow H_{per,e}^{2s}$ shows that $\widetilde{F}$ maps bounded and closed sets in $H_{per,e}^{2s}\times (\tfrac{1}{2},+\infty)$ into $H_{per,e}^{2s}$. Thus, if $R\subset S\subset H_{per,e}^{2s}\times (\tfrac{1}{2},+\infty)$ is a bounded and closed set, we obtain that $\widetilde{F}(R)$ is relatively compact in $H_{per,e}^{2s}$. Since $R$ is closed, any sequence $\{(\varphi_n,\omega_n)\}_{n\in\mathbb{N}}$ has a convergent subsequence in $R$, so $R$ is compact in $H_{per,e}^{2s}\times (\tfrac{1}{2},+\infty)$.\\
\indent Since the frequency of the wave given by $(\ref{varphi-stokes3})$ is not constant, we can apply \cite[Theorem 9.1.1]{buffoni-toland} to extend globally the local bifurcation curve given in \eqref{localcurve}. More precisely, there is a continuous mapping
\begin{equation}\label{globalcurve}
	\omega \in \left( \tfrac{1}{2}, +\infty \right) \longmapsto \varphi_{\omega} \in H_{per,e}^{2s}
\end{equation}
where $\varphi_{\omega}$ solves equation $(\ref{EDO1})$.
\end{proof}
\begin{remark}\label{remstokes} It is important to note that $\varphi \in H^{2s}_{per,e}$ given by \eqref{varphi-stokes2} is a solution of the minimization problem \eqref{minimizationproblem} by using similar arguments  as in \cite[Lemma 2.3]{NataliLePelinovsky}. \end{remark}

\section{Spectral analysis and uniqueness of minimizers}\label{spectralanalysis-section}

\subsection{Spectral analysis} We are going to use the variational characterization determined in the last section to obtain useful spectral properties for the linearized operator $\mathcal{L}$ in $(\ref{matrixop})$ around the periodic wave $\varphi$ obtained by Theorem $\ref{theorem1even}$.\\
\indent Let $s \in \left(\tfrac{1}{4},1\right]$ and $\omega>0$ be fixed. Consider $\varphi \in H^{\infty}_{per,e}$ the periodic minimizer obtained by Theorem \ref{theorem1even}. We study the spectral properties of the matrix operator
$$
\mathcal{L}=\left(\begin{array}{cccc}\mathcal{L}_1 & 0\\
	0 & \mathcal{L}_2\end{array}\right):H^{2s}_{per} \times H^{2s}_{per} \subset L^2_{per} \times L^2_{per} \rightarrow L^2_{per} \times L^2_{per},
$$
where $\mathcal{L}_1, \mathcal{L}_2$ are defined by
\begin{equation}\label{L1eL2}
	\mathcal{L}_1=(-\Delta)^s+\omega-3\varphi^2
	\qquad \text{and} \qquad
	\mathcal{L}_2=(-\Delta)^s+\omega-\varphi^2.
\end{equation}

The operators $\mathcal{L}_1$ and $\mathcal{L}_2$ are the real and imaginary parts of the operator
$\mathcal{L}$. An important fact is that by \eqref{positivity}, we have
$$
G(u) \leq \mathcal{B}_\omega (u),
$$
Next, since $\varphi$ is a minimizer of $\mathcal{B}_\omega$, we conclude that $\varphi$ also is a minimizer of $G$ in \eqref{G}. By \cite[Theorem 30.2.2]{BlanchardBruning}, we infer
$$
\mathcal{L}_{|_{ \{\varphi^3  \}^{\perp}}} \geq 0,
$$
and by Min-Max Principle in \cite[Theorem XIII.2]{ReedSimon} we obtain
$
{\rm n}(\mathcal{L}) \leq 1.
$
Since $\mathcal{L}_1\varphi=-2\varphi^3$ and $(\mathcal{L}_1\varphi,\varphi)_{L^2}=-2\int_{-\pi}^{\pi}\varphi^4dx<0$, we deduce by the Min-Max Principle that ${\rm n}(\mathcal{L}_1) \geq 1$. Since the operator $\mathcal{L}$ in $(\ref{matrixop})$ is diagonal, we also obtain ${\rm n}(\mathcal{L}_1)={\rm n}(\mathcal{L})=1$, so that ${\rm n}(\mathcal{L}_2)=0$. Next, we see that $\mathcal{L}_2\varphi=0$ with ${\rm n}(\mathcal{L}_2)=0$. It follows by oscillation theorem in \cite{HurJohnsonMartin} that $\varphi>0$ and the zero eigenvalue for $\mathcal{L}_2$ results to be simple. In fact, we have proved the following result.

\begin{lemma}\label{simpleKernel2even}
Let $s \in \left( \tfrac{1}{4},1 \right]$ and $\omega>\frac{1}{2}$ be fixed. If $\varphi \in H^\infty_{per,e}$ is the periodic minimizer given by Theorem \ref{theorem1even}, then  ${\rm n}(\mathcal{L}_2)=0$ and  ${\rm z}(\mathcal{L}_2)=1$.
\end{lemma}

Concerning the operator $\mathcal{L}_1$ in \eqref{L1eL2}, we have the following lemma.

\begin{lemma}\label{simpleKerneleven}
Let $s \in \left( \tfrac{1}{4},1 \right]$ and $\omega>\frac{1}{2}$ be fixed. If $\varphi \in H^\infty_{per,e}$ is the periodic minimizer given by Theorem \ref{theorem1even}, then ${\rm Ker}(\mathcal{L}_1)=[\varphi']$.
\end{lemma}
\begin{proof}
First, we see that $\varphi' \in {\rm Ker}(\mathcal{L}_1)={\rm R}(\mathcal{L}_1)^{\perp}$. In addition, since $\mathcal{L}_1\, \varphi = -2\varphi^3 $, we obtain
$
\varphi^3 \in {\rm R}(\mathcal{L}_1).
$
Next,  we claim that $\varphi \in  {\rm R}(\mathcal{L}_1)$ and to do so, we follow the notations contained in the Appendix A. Indeed, we start by defining
$
e_i:= \sqrt{|\lambda_i(0)|}\,\widehat{\xi_i}, \; i \in \mathbb{N}.
$

From Lemma \ref{prop314-natali}, we have that $(e_i)_{i \in \mathbb{N}} \subset X$ forms an orthonormal complete system in $X$, composed by eigenfunctions of $S_0: X \rightarrow X$. Also, we define
\begin{equation}\label{eta1}
\eta_1:= \sum_{i \in \mathbb{N}\setminus N} \left( \frac{1}{1-\lambda_i(0)}\right) \left( \frac{\widehat{\varphi}}{w_0}, e_i\right)_{X,0}e_i,
\end{equation}
where $N:=\{ i \in \mathbb{N} \; ; \; \lambda_i(0)=1\}$.  By \cite[Corollary 3.1]{AnguloNatali2008}, we note that $N \subset \mathbb{N}$ is finite, since ${\rm z}(\mathcal{L}_1) \in \{1,2\}$. Thus, we can write
\begin{eqnarray}\label{est1}
\sum_{i \in \mathbb{N}\setminus N} \left( \frac{1}{1-\lambda_i(0)}\right)^2 \bigg| \left( \frac{\widehat{\varphi}}{w_0}, e_i\right)_{X,0} \bigg|^2  \leq  A\, \sum_{i \in \mathbb{N}\setminus N} \bigg| \left( \frac{\widehat{\varphi}}{w_0}, e_i\right)_{X,0}\bigg|^2  \leq  A\, \bigg\| \frac{\widehat{\varphi}}{w_0}\bigg\|^2_{X,0},
\end{eqnarray}
where
$
A:=\sup_{i \in \mathbb{N}\setminus N} \bigg|\frac{1}{1-\lambda_i(0)}\bigg|^2 \in [0,\infty).
$
However, by Parseval Identity, we see that
$$
\bigg\| \frac{\widehat{\varphi}}{w_0}\bigg\|^2_{X,0}= \sum_{n \in \mathbb{N}}\bigg| \frac{\widehat{\varphi}(n)}{w_0(n)}\bigg|^2 w_0^2(n)= \sum_{n \in \mathbb{N}} | \widehat{\varphi}(n)|^2=\|\varphi\|^2_{L^2},
$$
which implies by $(\ref{est1})$
$$
\sum_{i \in \mathbb{N}\setminus N} \left( \frac{1}{1-\lambda_i(0)}\right)^2 \bigg| \left( \frac{\widehat{\varphi}}{w_0}, e_i\right)_{X,0} \bigg|^2  \leq A \|\varphi \|^2_{L^2}< \infty.
$$
Hence, the series for $\eta_1$ in $(\ref{eta1})$ converges in $X$, so that $\eta_1 \in X$. For $\theta=0$, we see that $\mathcal{F}:Y\longrightarrow X$ defined as $\mathcal{F}(f)=\widehat{f}$ is an isomorphism. Here, $\mathcal{F}$ denotes the Fourier transform defined in ${L}_{per}^2$. Thus, one can take $\psi_1 \in Y \hookrightarrow {L}^2_{per}$ such that $\widehat{\psi}_1=\eta_1$. Moreover, by the definition of $X$ for $\theta=0$ and the fact that $s\in\left(\frac{1}{4},1\right]$, it follows that
$$
\|\psi_1\|_{H^{2s}_{per}}^2= 2\pi \sum_{n \in \mathbb{Z}} \big(1+|n|^2\big)^{2s} |\widehat{\psi_1}(n)|^2=2\pi \sum_{n \in \mathbb{Z}} \big(1+|n|^2\big)^{2s} |\eta_1 (n)|^2< \infty,
$$
that is, $\psi_1 \in H^{2s}_{per}=D(\mathcal{L}_1)$.  Thus, we have
\begin{equation}\label{fourier0}\begin{array}{llllllll}
\widehat{(\mathcal{L}_1\, \psi_1)}&=&\left( (-\Delta)^s\psi_1+\omega \psi_1-3\varphi^2 \psi_1  \right)\widehat{}  \\
&= & \widehat{(-\Delta)^s\psi_1}+\omega\widehat{\psi_1}-\widehat{3\varphi^2 \psi_1 }  \\
&=& \beta \widehat{\psi_1}+\omega \widehat{\psi_1}-\widehat{(3\varphi^2\psi_1)} =(\beta+\omega)\widehat{\psi_1}-\widehat{(3\varphi^2\psi_1)} r \\
& = &  (\beta+\omega)\widehat{\psi_1}-w_0 \big(\big((-\Delta)^s+\omega \big)^{-1}\, 3\varphi^2 \psi_1 \big)\,\widehat{}   \\
& =& w_0 \widehat{\psi_1} - w_0 ( T_0 (\psi_1) )\,\widehat{}= w_0  \eta_1 - w_0 ( T_0 (\widecheck{\eta}_1) )\,\widehat{}.
\end{array}\end{equation}

On the other hand,
\begin{equation}\label{fourier1} \begin{array}{lllll}
 T_0 (\widecheck{\eta}_1) & = & \displaystyle T_0 \left(  \sum_{i \in \mathbb{N}\setminus N} \left( \frac{1}{1-\lambda_i(0)}\right) \left( \frac{\widehat{\varphi}}{w_0}, e_i\right)_{X,0} \widecheck{e_i} \right)  \\
& = & \displaystyle T_0 \left(  \sum_{i \in \mathbb{N}\setminus N} \left( \frac{1}{1-\lambda_i(0)}\right) \left( \frac{\widehat{\varphi}}{w_0}, e_i\right)_{X,0} \sqrt{|\lambda_i(0)|}\xi_i \right)   \\
& = & \displaystyle  \sum_{i \in \mathbb{N}\setminus N} \left( \frac{1}{1-\lambda_i(0)}\right) \left( \frac{\widehat{\varphi}}{w_0}, e_i\right)_{X,0} \sqrt{|\lambda_i(0)|} \; T_0 (\xi_i).
\end{array}\end{equation}
\indent A straightforward calculation using the relation $e_i= \sqrt{|\lambda_i(0)|}\,\widehat{\xi_i}$ also gives
\begin{eqnarray}\label{fourier2}
S_0( e_i )= \lambda_i(0) e_i,
\end{eqnarray}
for all $i \in \mathbb{N}$.


By \eqref{eigenvalueT}, \eqref{fourier1} and \eqref{fourier2}, we infer
\begin{equation}\label{fourier3}\begin{array}{llllllll}
\big( T_0 (\widecheck{\eta}_1) \big)\,\widehat{} & =  & \displaystyle \sum_{i \in \mathbb{N}\setminus N} \left( \frac{1}{1-\lambda_i(0)}\right) \left( \frac{\widehat{\varphi}}{w_0}, e_i\right)_{X,0} \sqrt{|\lambda_i(0)|} \; \big(T_0 (\xi_i)\big)\, \widehat{}  \\
& =  & \displaystyle \sum_{i \in \mathbb{N}\setminus N} \left( \frac{1}{1-\lambda_i(0)}\right) \left( \frac{\widehat{\varphi}}{w_0}, e_i\right)_{X,0} \sqrt{|\lambda_i(0)|} \; \lambda_i(0) \widehat{\xi_i}  \\
& =  &  \displaystyle\sum_{i \in \mathbb{N}\setminus N} \left( \frac{1}{1-\lambda_i(0)}\right) \left( \frac{\widehat{\varphi}}{w_0}, e_i\right)_{X,0} \sqrt{|\lambda_i(0)|} \; S_0(\widehat{\xi_i})   \\
& =  & \displaystyle \sum_{i \in \mathbb{N}\setminus N} \left( \frac{1}{1-\lambda_i(0)}\right) \left( \frac{\widehat{\varphi}}{w_0}, e_i\right)_{X,0}  S_0( e_i )   \\
& =  &  \displaystyle\sum_{i \in \mathbb{N}\setminus N} \left( \frac{1}{1-\lambda_i(0)}\right) \left( \frac{\widehat{\varphi}}{w_0}, e_i\right)_{X,0}   \lambda_i(0)\, e_i.
\end{array}\end{equation}

\indent Now, we need to adopt the notational convention $0=\tfrac{0}{0}$ to obtain by \eqref{eta1}, \eqref{fourier0} and \eqref{fourier3} that
\begin{equation}\label{fourier4}\begin{array}{llllll}
\widehat{(\mathcal{L}_1\psi_1)}&=& \displaystyle w_0 \sum_{i \in \mathbb{N}\setminus N} \left( \frac{1}{1-\lambda_i(0)}\right) \left( \frac{\widehat{\varphi}}{w_0}, e_i\right)_{X,0} \Big(1-\lambda_i(0)\Big)e_i \\
&=&\displaystyle w_0 \sum_{i \in \mathbb{N}}  \left( \frac{\widehat{\varphi}}{w_0}, e_i\right)_{X,0} e_i=\frac{w_0}{w_0} \widehat{\varphi} =\widehat{\varphi}.\end{array}
\end{equation}

Applying the inverse Fourier transform in \eqref{fourier4}, it follows that $\mathcal{L}_1 \psi_1=\varphi$. In other words,
$
\varphi \in {\rm R}(\mathcal{L}_1),
$
as claimed.\\
\indent Thus, we obtain that $\varphi, \varphi^3 \in {\rm R}(\mathcal{L}_1)={\rm Ker}(\mathcal{L}_1)^{\perp}$ with $\varphi$ being an even, smooth, positive, and single-lobe solution for $(\ref{EDO1})$. Let us assume that ${\rm z}(\mathcal{L}_1)=2$. Since $\varphi' \in {\rm Ker}(\mathcal{L}_1)$ is odd, there exists an even periodic function $h \in {\rm Ker}(\mathcal{L}_1)$ such that $h$ has exactly two symmetric zeros  in the interval $\left[-\pi,\pi\right)$ (see oscillation theorem in \cite{HurJohnsonMartin}). Hence, there exist $x_0 \in \left(-\pi,\pi\right)$ such that $h(\pm x_0)=0$. Without loss of generality, we can still suppose that
\begin{equation}\label{positivityh}
h(x)>0, \; x \in (-x_0, x_0) \quad \text{and} \quad  h(x)<0, \; x \in\left [-\pi,x_0 \right) \cup \left(x_0,\pi\right).
\end{equation}

Furthermore, since $h\in \Ker(\mathcal{L}_1)$ and $\varphi,\varphi^3\in\Ker(\mathcal{L}_1)^{\bot}$ we have
\begin{equation}\label{orthogonalityL1}
(h,\varphi^3)_{L^2}=0 \quad \text{and} \quad (h,\varphi)_{L^2}=0.
\end{equation}

Since $\varphi>0$,  we obtain by the fact that $\varphi$ is a single-lobe that
$
\varphi(x)(\varphi(x)^2-\varphi^2(x_0))
$
is positive over $(-x_0,x_0)$ and negative over $\left [-\pi,x_0 \right) \cup \left(x_0,\pi\right)$, so that it has the same behaviour as $h$ in \eqref{positivityh}. Thus, $(\varphi(\varphi^2-\varphi^2(x_0)), h)_{L^2} \neq 0$ which leads a contradiction with \eqref{orthogonalityL1}. Consequently, we have ${\rm Ker}(\mathcal{L}_1)=[\varphi']$.
\end{proof}

\begin{remark}
	Arguments established in the end of the proof of Lemma $\ref{simpleKernel2even}$ are valid only if $\varphi>0$. If $\varphi$ changes its sign over $\mathbb{T}$, an alternative form to prove that $\Ker(\mathcal{L}_1)=[\varphi']$ can be determined by proving that $1\in {\rm R}(\mathcal{L}_1)$. In the affirmative case and since $\mathcal{L}_11=\omega-3\varphi^2$, we obtain that the property $\{1,\varphi,\varphi^2\}\subset {\rm R}(\mathcal{L}_1)$ occurs. Employing the arguments in \cite[Proposition 2.5]{NataliLePelinovskyKDV}, we obtain that ${\rm Ker}(\mathcal{L}_1)=[\varphi']$ as requested.
\end{remark}

As a consequence of Lemma \ref{simpleKerneleven}, we obtain the existence of a smooth curve of positive and periodic solutions $\varphi_{\omega}$ depending on the wave frequency $\omega>0$ all of the with the same period $2\pi$.

\begin{proposition}\label{smoothcurveeven}
Let $s \in \left(\tfrac{1}{4},1\right]$ and $\varphi_0 \in H^{\infty}_{per}$ be the solution obtained in the Proposition \ref{theorem1even} which is associated to the fixed value $\omega_0>\frac{1}{2}$. Then, there exists a $C^1$ mapping
$$
\omega \in \mathcal{I}_{\omega_0}  \longmapsto \varphi_\omega \in H^s_{per,e}
$$
defined in an open neighbourhood $\mathcal{I}_{\omega_0} \subset (0,+\infty)$ of $\omega_0>0$ such that $\varphi_{\omega_0}=\varphi_0$.
\end{proposition}
\begin{proof}
The proof follows from the implicit function theorem and it is similar to \cite[Theorem 3.2]{Cristofani}.
\end{proof}

\begin{remark}
We cannot guarantee that for each $\omega \in \mathcal{I}_{\omega_0}$ given by Proposition \ref{smoothcurveeven} that $\varphi_{\omega}$ solves the minimization problem \eqref{minimizationproblem} except at $\omega=\omega_0$.
\end{remark}

\indent The results determined in this subsection can be summarized in the following proposition:

\begin{proposition}\label{propspec}
	Let $\varphi$ be the single-lobe profile obtained in Proposition $\ref{theorem1even}$. We have that $n(\mathcal{L})=1$ and $\Ker(\mathcal{L})=[(\varphi',0),(0,\varphi)]$.
\end{proposition}
\begin{flushright}
	$\blacksquare$
\end{flushright}

\subsection{Uniqueness of real minimizers}\label{uniquenesseven}

In this subsection we  show the uniqueness for the real periodic minimizers $\varphi$ obtained in Theorem \ref{theorem1complex}. To this end, we proceed as in \cite[Section 3.2]{AmaralBorluk}. The main difference to our approach is that we do not need to assume that the kernel of the linearized operator restricted to the space of zero mean periodic waves are simple. First, the space of zero mean periodic waves is not suitable for our purposes since we are working with real positive periodic waves $\varphi$. The equivalent condition in our case would be assuming that ${\rm Ker}(\mathcal{L}_1)=[\varphi']$ for every $\omega \in \left(\tfrac{1}{2}, +\infty\right)$ and $s \in \left( \tfrac{1}{4},1 \right)$. However, we have already determined this property in Lemma \ref{simpleKerneleven} and therefore, we are in conformity with the arguments proved in \cite[Section 5]{FrankLenzmann} where the authors have established the uniqueness of solitary waves arising as minimizers of a similar problem as in $(\ref{minimization1})$.\\
\indent In what follows, let us define the complex Banach space
$$
\mathbb{V}:= \{f=f_1+if_2\equiv(f_1,f_2) \in {L}^4_{per}\times L_{per}^4 \; ; \; f_1,f_2 \in L^4_{per,e}\},
$$
endowed with the norm of ${L}^4_{per}$. We have the following result:

\begin{prop}\label{prop1-uniq}
	Let $s_0\in(\frac{1}{2}, 1)$. Suppose that
	$(\varphi_{0}+0i, \omega_{0})\in \mathbb{V}\times \left(\frac{1}{2},+\infty\right)$ where $\varphi_0$ is a non-zero real solution of $(\ref{EDO1})$ with $s=s_0$ and $\omega=\omega_{0}$. Then, for some $\delta>0$, there exists a $C^1-$map
	$s\in I\rightarrow (\varphi_{s}+0i,\omega_{s})\in \mathbb{V}\times \left(\frac{1}{2},+\infty\right)$, defined in the interval $I=[s_0, s_0+\delta)$, such that the following
	holds:\\
	\begin{itemize}
		\item[(i)] $(\varphi_{s}+0i, \omega_{s})$ solves the equation $(\ref{EDO1})$ with $\omega=\omega_{s}$, for all $s\in I$;
		\item[(ii)] There exists $\varepsilon>0$ such that $(\varphi_{s}+0i, \omega_{s})$ is the unique solution of $(\ref{EDO1})$ for $s\in I$ in
		the neighbourhood $\{(\varphi+0i,\omega)\in \mathbb{V}\times \left(\frac{1}{2},+\infty\right);\ ||\varphi-\varphi_{0}||_{\mathbb{V}}+|\omega-\omega_{0}|<\varepsilon\}$;
		\item[(iii)] For all $s\in I$, we have $ \int_{-\pi}^{\pi}\varphi_{s}^4dx=\int_{-\pi}^{\pi}\varphi_{0}^4dx.$
	\end{itemize}
\end{prop}
\begin{proof}
The proof is similar to the one given in \cite[Proposition 5]{AmaralBorluk} therefore we  only  give the main steps. Indeed, let $s_0 \in \left(\tfrac{1}{2}, 1 \right)$ be fixed and consider $(\Phi_0,\omega_0):=(\varphi_0+i0,\omega_0) \in \mathbb{V} \times \left(\tfrac{1}{2},+\infty\right)$, where $\varphi_0 \in \mathbb{V}$ satisfies $(\ref{EDO1})$.

We define the mapping $\mathcal{F}: \mathbb{V} \times \left(\tfrac{1}{2}, +\infty\right) \times I_\delta \longrightarrow \mathbb{V} \times \mathbb{R}$ by
\begin{eqnarray}
\mathcal{F}(\Phi,\omega, s)=\begin{bmatrix}
\Phi-((-\Delta)^s+\omega)^{-1}|\Phi|^2\Phi \\
\|\Phi\|_{{L}_{per}^4}^4-\|\Phi_0\|_{{L}_{per}^4}^4
\end{bmatrix},
\end{eqnarray}
where $I_\delta:=[s_0, s_0+\delta)$ with $\delta>0$ will be chosen later. We note that $\mathcal{F}$ is a  well-defined $C^1$- mapping (\cite[Lemma E.1]{FrankLenzmann}) and $\mathcal{F}(\Phi_0,\omega_0, s_0)=(0,0)$.

In particular,  we see that the Fréchet derivative of $\mathcal{F}$ with respect $(\Phi,\omega)$ at $(\Phi_0,\omega_0,s_0)$ is given by
$$
D_{\Phi,\omega}\mathcal{F}(\Phi_0,\omega_0,s_0)=\begin{bmatrix}
1-((-\Delta)^{s_0}+\omega_0)^{-1} \, 3\Phi_0^2 & ((-\Delta)^{s_0}+\omega_0)^{-2}\Phi_0^2  \\
4(\Phi_0^3, \cdot)_{{L}_{per}^2} & 0
\end{bmatrix}.
$$

Since $\varphi_0'$ is odd and ${\rm Ker}(\mathcal{L}_1)=[\varphi_0']$, we can show that $D_{\Phi,\omega}\mathcal{F}(\Phi_0,\omega_0,s_0)$ is invertible. By implicit function theorem, we guarantee the existence of a $C^1$-map
\begin{equation}\label{localbranchfunction}
s \in I_{\delta} \longmapsto (\Phi_s, \omega_s) \in \mathbb{V} \times \left(\tfrac{1}{2}, +\infty \right),
\end{equation}
defined over $I_{\delta}$, where $\delta>0$ is small enough. Here $\Phi_s$ is defined in a neighbourhood of the point $\Phi_0=\varphi_0+0i\in \mathbb{V}$ and this fact enables us to define, without loss of generality, that $\Phi_s:=\varphi_s+0i\in \mathbb{V}$. Thus, we can consider a local branch of solutions $(\varphi_s+0i,\omega_s) \in \mathbb{V} \times \left(\tfrac{1}{2}, +\infty \right)$  for the equation \eqref{EDO1} and parametrized by $s \in I_\delta$.

\end{proof}
The next step is to consider the corresponding maximal extension of the branch \\ \mbox{$(\varphi_{s}, \omega_{s}):=(\varphi_s+0i,\omega_s)$} given by
$s\in [s_0, s_{*})$, where 
$$\begin{array}{rrrr}s_{*}:=\sup\{q;\ s_0<q<1,\  (\varphi_{s},\omega_{s})\in C^1([s_0, q);\mathbb{V}\times \left(\frac{1}{2},+\infty\right))\ \mbox{given by Proposition \ref{prop1-uniq}}\\\\
	\mbox{and}\ (\varphi_{s}, \omega_{s})\ \mbox{satisfies (\ref{EDO1}) for}\ s\in [s_0, q)\}.\end{array}$$
It is clear that $s_{*}\leq1$ and it makes necessary to prove $s_{*}=1$. To do so, we establish the following result:

\begin{prop}\label{prop-seq-ext}
	Let $\{s_n\}_{n=1}^{n=+\infty}\subset (\frac{1}{2}, s_{*})$ be a sequence such that $s_n\rightarrow s_{*}$. Furthermore,
	we assume that $\varphi_{s_n}\in \mathbb{V}$ are the corresponding solutions obtained in Proposition $\ref{prop1-uniq}$ with frequency of the wave given by $\omega_{s_{n}}$. Up to a subsequence, it follows that
	$$\varphi_{s_n}\rightarrow \varphi_{*}\ \mbox{in}\ L_{per}^p(\mathbb{T})\ \ \mbox{and}\ \ \omega_{s_n}\rightarrow \omega_{*},$$
for all $p\geq1$. Here, $\varphi_{*}$ satisfies equation the $(\ref{EDO1})$ where $\omega_{*}\in \left(\frac{1}{2},+\infty\right)$ is the corresponding frequency of the wave. Moreover, the corresponding maximal branch $(\varphi_{s}, \omega_{s})\in C^1([s_0, s_{*});\mathbb{V}\times\left(\frac{1}{2},+\infty\right))$ extends to $s_{*}=1$.
\end{prop}
\begin{proof}
	The proof of this result is similar to \cite[Proposition 6]{AmaralBorluk} and we omit the details.
\end{proof}
\begin{proposition}[Uniqueness of real minimizers]\label{UniquenessRealEven}
Let $s \in \left(\tfrac{1}{2}, 1 \right)$ be fixed. The real and even periodic minimizer obtained in Theorem \ref{theorem1complex} is unique.
\end{proposition}
\begin{proof}
It follows by similar arguments as in \cite[Proposition 7]{AmaralBorluk}.
\end{proof}

\section{Orbital Stability}\label{stability-section}

In this section, we present the orbital stability results. It is well known that (\ref{fNLS1}) has two basic symmetries, namely, translation and rotation. If $u = u(x,t)$ is a solution of (\ref{fNLS1}), so are $e^{-i\zeta}u$ and $u(x-r,t)$ for any  $\zeta, r \in \mathbb{R}$. Considering $u=P+iQ\equiv(P,Q)$, we obtain that \eqref{fNLS1} is invariant under the transformations
\begin{equation}\label{T1}
	S_1(\zeta)u := \left(
	\begin{array}{cc}
		\cos{\zeta} & \sin{\zeta} \\
		- \sin{\zeta} & \cos{\zeta}
	\end{array}
	\right) \left(
	\begin{array}{c}
		P \\ Q
	\end{array}
	\right)
\end{equation}
and
\begin{equation}\label{T2}
	S_2(r)u := \left(
	\begin{array}{c}
		P(\cdot - r, \cdot) \\
		Q(\cdot - r, \cdot)
	\end{array}
	\right).
\end{equation}
The actions $S_1$ and $S_2$ define unitary groups in ${H}^s_{per}$ with infinitesimal generators given by
$
	S_1'(0)u := \left(
	\begin{array}{cc}
		0 & 1 \\
		-1 & 0
	\end{array}
	\right) \left(
	\begin{array}{c}
		P \\
		Q
	\end{array}
	\right) = J \left(
	\begin{array}{c}
		P \\
		Q
	\end{array}
	\right)
$
and
$
	S_2'(0)u := \partial_x \left(
	\begin{array}{c}
		P \\
		Q
	\end{array}
	\right).
$

A standing wave solution  as in (\ref{standingwave}) is given by
$
	u(x,t) = e^{i\omega t}\varphi(x)=\left(
	\begin{array}{c}
		\varphi(x) \cos(\omega t) \\
		\varphi(x) \sin(\omega t)
	\end{array}
	\right).
$
Since the equation $(\ref{fNLS1})$ is invariant under the actions of  $S_1$ and $S_2$, we define the orbit generated by
$
\Phi = (\varphi,0)
$
as
\begin{equation*}
	\mathcal{O}_\Phi = \big\{S_1(\zeta)S_2(r)\Phi; \zeta, r \in \R \big\} = \left\{ \left(
	\begin{array}{cc}
		\cos{\zeta} & \sin{\zeta} \\
		-\sin{\zeta} & \cos{\zeta}
	\end{array}
	\right) \left(
	\begin{array}{c}
		\varphi(\cdot - r) \\
		0
	\end{array}
	\right); \; \zeta, r \in \R \right\}.
\end{equation*}

The pseudometric $d$ in ${H}^s_{per}$ is given by
$
	d(f,g):= \inf \{ \|f - S_1(\zeta)S_2(r)g\|_{{H}_{per}^s} ; \; \, \zeta, r \in \R\}.
$
The distance between $f$ and $g$ is the distance between $f$ and the orbit generated by $g$ under the action of rotation and translation, so that
$
d(f,\Phi) = d(f,\mathcal{O}_\Phi).
$

We now present our notion of orbital stability.

\begin{definition}\label{defstab}
  Let $\Theta(x,t) = (\varphi(x)\cos(\omega t), \varphi(x) \sin(\omega t))$ be a standing wave for (\ref{fNLS1}). We say that $\Theta$ is orbitally stable in ${H}^s_{per}$ provided that, given $\varepsilon > 0$, there exists $\delta > 0$ with the following property: if $u_0 \in {H}^s_{per}$ satisfies $\|u_0 - \Phi\|_{{H}^s_{per}} < \delta$, then the local solution $u(t)$ defined in the semi-interval $[0,+\infty)$ satisfies
	$
		d(u(t), \mathcal{O}_\Phi) < \varepsilon$,  for all  $t \geq 0.
	$
	Otherwise, we say that $\Theta$ is orbitally unstable in ${H}^s_{per}$.
\end{definition}

\textit{Proof of Theorem $\ref{mainth}$}. By Proposition $\ref{propspec}$, we see that $n(\mathcal{L})=1$ and ${\rm Ker}(\mathcal{L})=[(\varphi',0),(0,\varphi)]$ and these two basic facts are crucial to determine results of orbital stability/instability for periodic waves. Since both spectral properties are valid, the proof of orbital stability follows similarly as in \cite[Theorem 4.17]{natalipastor} but we need to take into account the result of global well-posedness as in Proposition $\ref{gwp}$ to prove the stability in terms of the two symmetries defined for the orbit $\mathcal{O}_{\Phi}$. For the orbital stability, we need to consider the Vakhitov–Kolokolov condition $\mathsf{q}>0$ which is equivalent to consider $(\mathcal{L}_1\Psi,\Psi)_{L^2_{per}}<0$, where $\Psi=-\frac{d}{d\omega}\varphi$, $\mathcal{L}_1\Psi=\varphi$ and  $(\mathcal{L}_1\Psi,\varphi')_{L^2_{per}}=0$. For the orbital instability in ${H}_{per,e}^s$, we first use  the approach in \cite{GrillakisShatahStraussI} and the condition $\mathsf{q}<0$ by considering the orbit $\mathcal{O}_{\Phi}$ having only one basic symmetry (namely, the orbit generated by the rotations only). As far as we know, the theory in \cite{GrillakisShatahStraussI} only requires that ${\rm n}(\mathcal{L})=1$ and ${\rm z}(\mathcal{L})=1,$ so that we need to remove out one of the symmetries in Definition $\ref{defstab}$. Since the space ${H}_{per,e}^s$ is not invariant under translation and $\varphi'$ is odd, the pair $(\varphi',0)$ can not be considered as an element of the subspace $\Ker(\mathcal{L})$ and thus, under this restriction, we have ${\rm Ker}(\mathcal{L}|_{L_{per,e}^2})=[(0,\varphi)]$. Here $\mathcal{L}|_{L_{per,e}^2}$ denotes the restriction of $\mathcal{L}$ over the subspace of even functions $L_{per,e}^2$. It is clear that if the  standing wave is orbitally unstable in a subspace ${H}_{per,e}^s$ of ${H}_{per}^s$,  then  it will also be unstable in the whole energy space ${H}_{per}^s$. The numerical approach determined below will be useful to decide the values of $s\in(\tfrac{1}{4},1]$ for which $\mathsf{q}>0$ or $\mathsf{q}<0$ in order to prove the orbital stability/instability.
\begin{flushright}
	$\blacksquare$
\end{flushright}

\subsection{Numerical experiments}

In this section we generate the periodic standing wave solutions of the fNLS equation by using the Petviashvili's iteration method.
The method is widely used for the generation of travelling wave solutions (\cite{aduran,AD,LP,OBM,PS}) and, specifically in the case of fNLS equation, some numerical studies have been also determined in \cite{klein}. Besides providing a numerical method in order to present the periodic single-lobe profile $\varphi$, our intention is to determine the sign of the quantity: \begin{equation}\label{q}\textsf{q}=\frac{d}{d\omega}\int_{-\pi}^{\pi} \varphi^2 dx.
	\end{equation}

 \begin{figure}[h]
\hspace{10pt}
 \begin{minipage}[t]{0.4\linewidth}
  \includegraphics[width=3.1in]{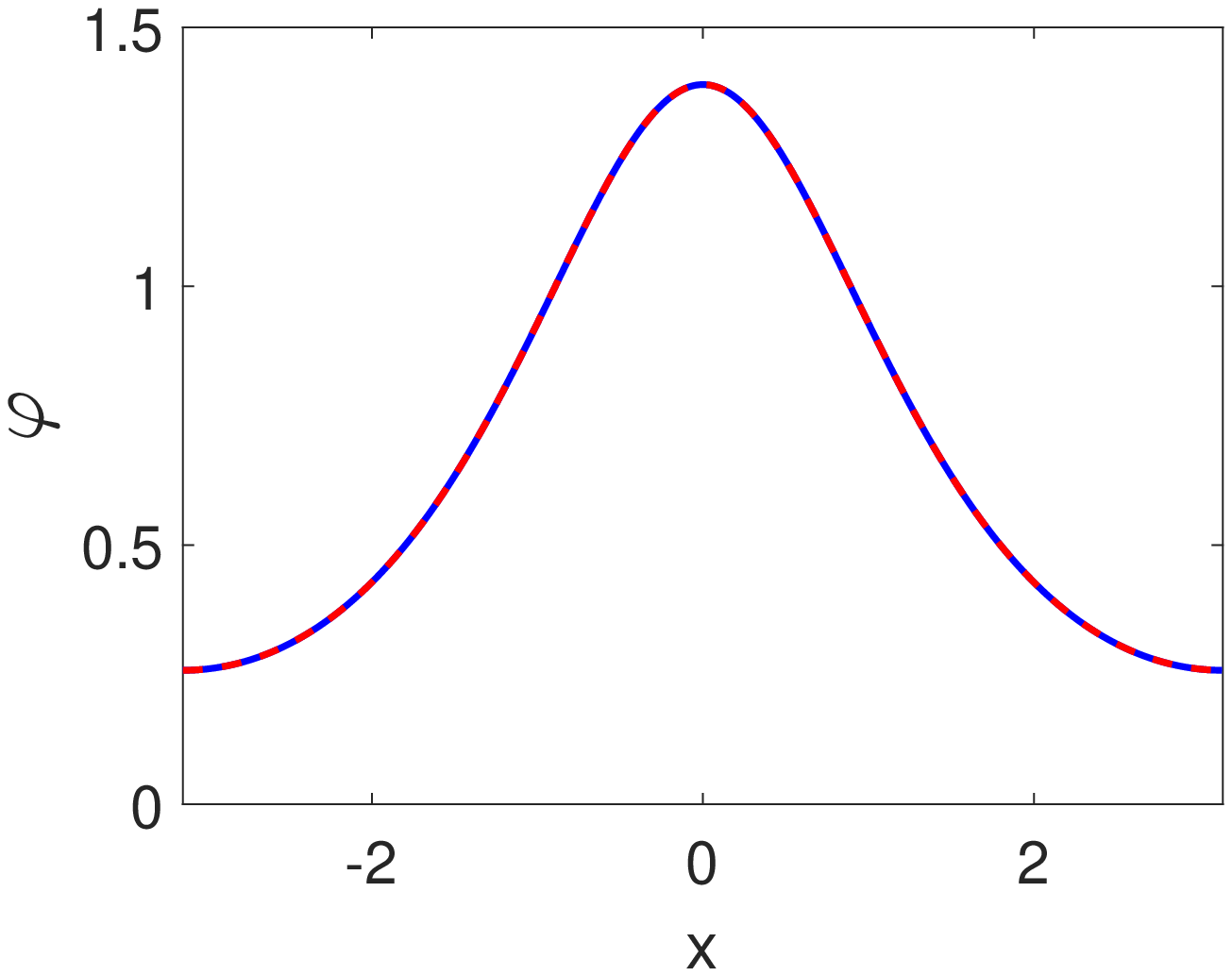}
 \end{minipage}
\hspace{30pt}
\begin{minipage}[t]{0.40\linewidth}
   \includegraphics[width=3.2in]{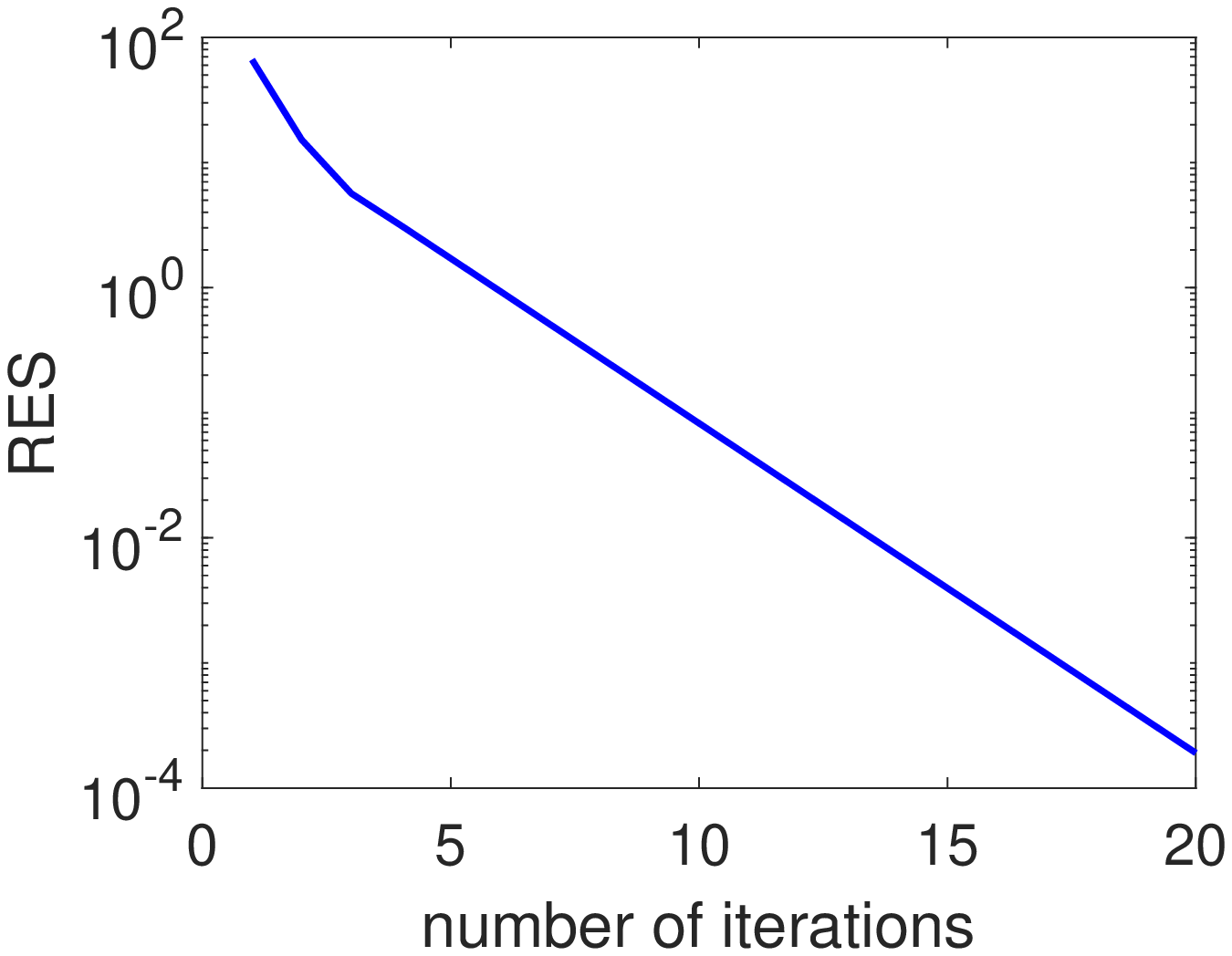}
 \end{minipage}
 \hspace{30pt}
\begin{minipage}[t]{0.40\linewidth}
   \includegraphics[width=3.3in]{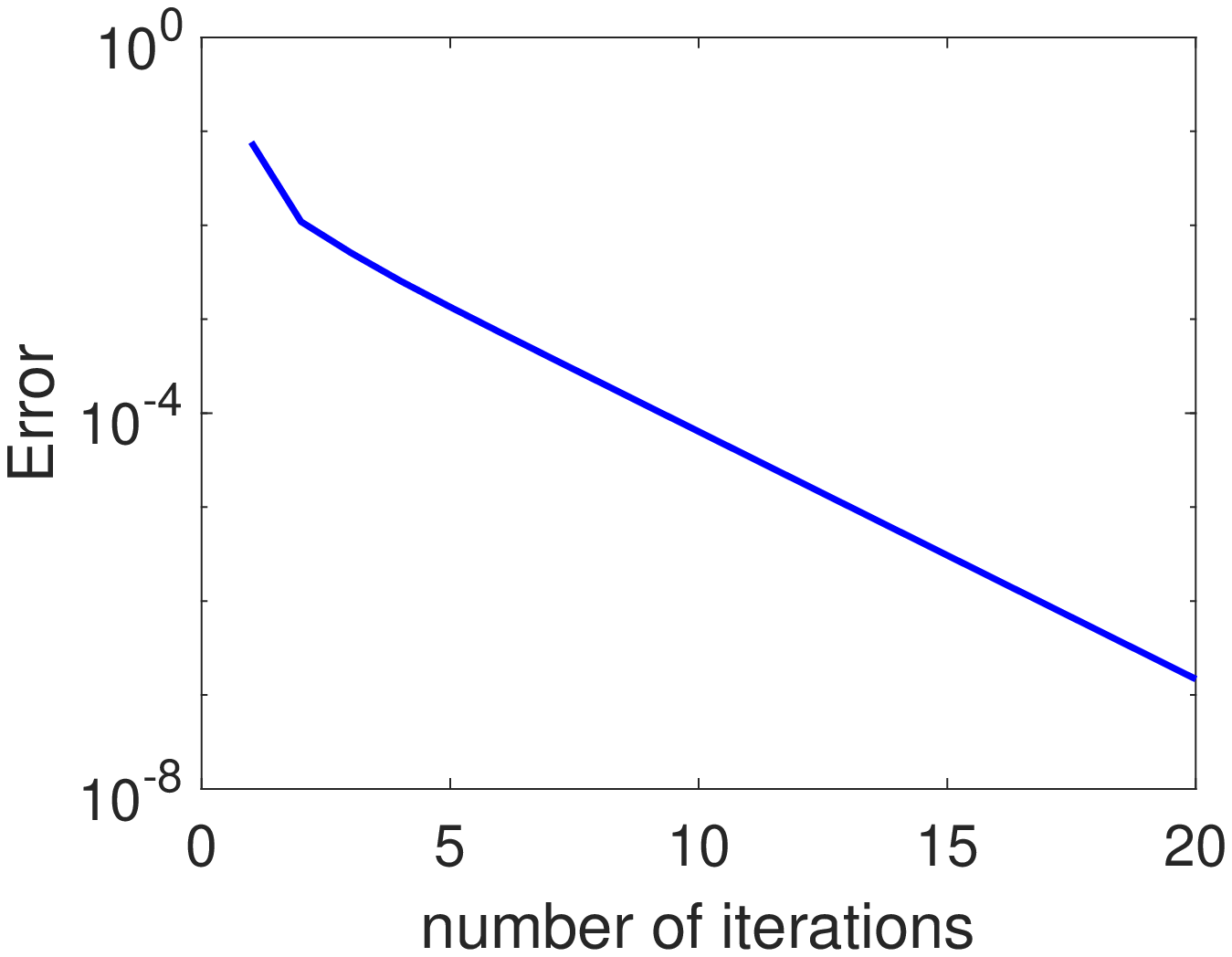}
 \end{minipage}
 \hspace{40pt}
\begin{minipage}[t]{0.40\linewidth}
   \includegraphics[width=3.3in]{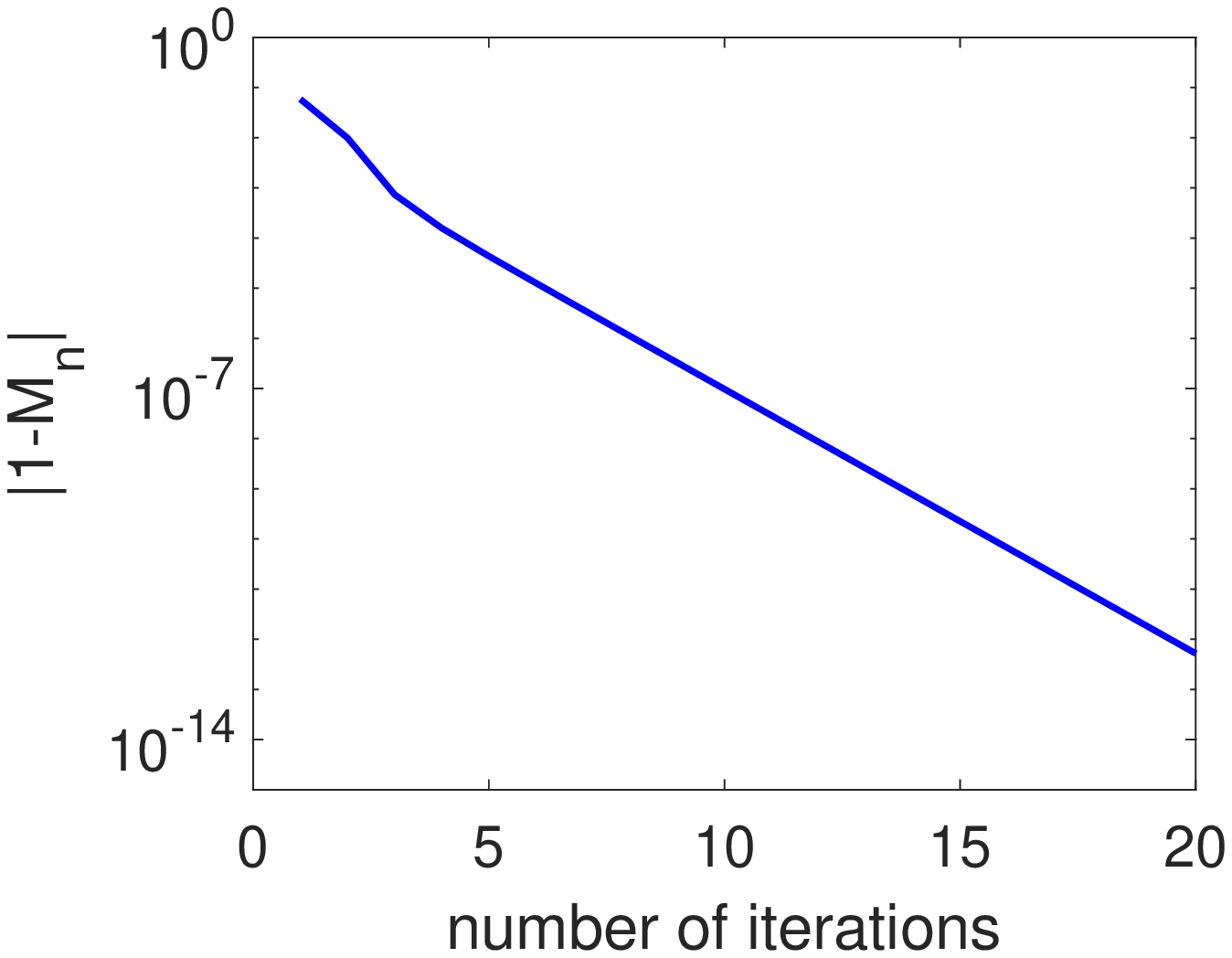}
 \end{minipage}
  \caption{The exact and the numerical solutions of  the fNLS equation with the wave frequency $\omega=1$  and  the variation of  \mbox{Error$(n)$}, $|1-M_n|$ and $RES$ with the number of iterations in semi-log scale.}
 \label{dif_er}
\end{figure}

\noindent  Applying the Fourier transform to the equation \eqref{EDO1} gives
\begin{equation}\label{iterate1}
\left(  |\xi|^{2s} + \omega  \right) \widehat{\varphi}(\xi)-  \widehat{\varphi^3}(\xi)=0, \quad \xi \in \mathbb{Z}.
\end{equation}
An iterative algorithm for numerical calculation of $\widehat{\psi}(\xi)$ for the equation \eqref{iterate1} can be proposed in the form
\begin{equation}\label{iterative1}
\widehat{\varphi}_{n+1}(\xi)=\frac{ \widehat{\varphi^3_n}(\xi)  }
{|\xi|^{2s} + \omega},\hspace{30pt} n\in\mathbb{N}
\end{equation}
where $\widehat{\varphi}_n(\xi)$ is the Fourier transform of ${\varphi}_n$ which is the $n^{\text{th}}$ iteration of the numerical solution. Here the solutions are constructed under the assumption
\begin{equation}\label{condition}
 |\xi|^{2s} + \omega \neq 0.
 \end{equation}

\noindent Since the above algorithm is usually divergent, we finally present the Petviashvilli's method as
\begin{equation}\label{iterative2}
\widehat{\varphi}_{n+1}(\xi)=\frac{(M_n)^{\nu}} {|\xi|^{2s} + \omega} \widehat{\varphi^3_n}(\xi)
\end{equation}
by introducing the stabilizing factor
\begin{equation}\label{sf}
  M_n=\frac{\big(({(-\Delta)}^{s} +\omega )\varphi_n, \varphi_n \big)_{L^2_{per}}}{(\varphi^3_n, \varphi_n)_{L^2_{per}}},~~~~~~~~~~ \hspace*{20pt} \varphi_n\in H^{2s}_{per}(\mathbb{T})
\end{equation}
Here, the free parameter $\nu$ is chosen as $1.5$ for the fastest convergence.
\noindent
The iterative process is controlled by the error between two consecutive iterations given by
$$
  \mbox{Error}(n)=\|\varphi_n-\varphi_{n-1}\|_{L_{per}^{\infty}}
$$
and the stabilization factor error given by
$
|1-M_n|. $ The residue of the interaction process is determined by
$
{RES(n)}= \|{ \mathcal{S}} \varphi_n\|_{L_{per}^{\infty}},
$
where
$
{ \mathcal{S}}\varphi={(-\Delta)}^{s} \varphi +\omega \varphi -\varphi^3.  $

\begin{figure}[!htbp]
\includegraphics[width=3.8in]{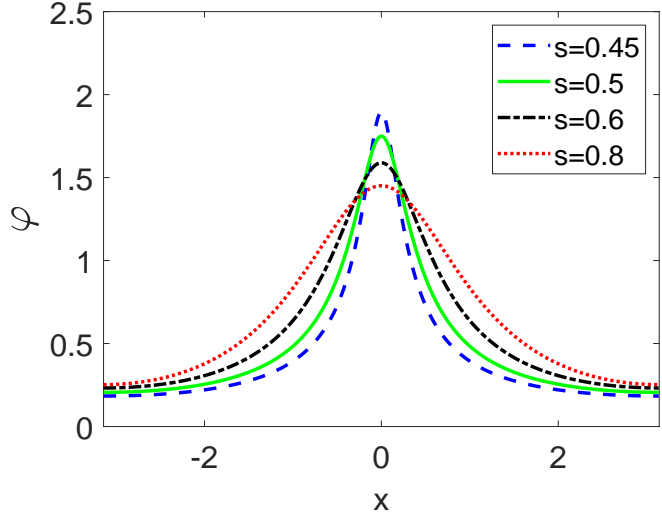}
  \caption{Numerical wave profiles for various values of $s \in (0,1)$ where $\omega=1$.}
 \label{waves}
\end{figure}

\indent The periodic standing wave solution of the fNLS equation  with $s=1$  is given in \cite{Angulo} as

\begin{equation}\label{dnsol}
  \varphi(x)=\eta_1 \mbox{dn}\left( \frac{\eta_1}{\sqrt{2}}x;\kappa \right),
\end{equation}
 where
  $ \kappa^2=\frac{\eta_1^2-\eta_2^2}{\eta_1^2},~~~~~~~~~~~\eta_1^2-\eta_2^2=2\omega, ~~~~~~~~~~~0<\eta_2<\eta_1$.
 Here the fundamental period is  $\displaystyle T_{\phi_\omega}= \frac{2\sqrt{2}}{\eta_1} {\rm K}(\kappa) $ where ${\rm K}(\kappa)$ is the complete elliptic integral of first kind.

\begin{figure}[h]
\hspace{-30pt}
 \begin{minipage}[t]{0.4\linewidth}
  \includegraphics[width=3.1in]{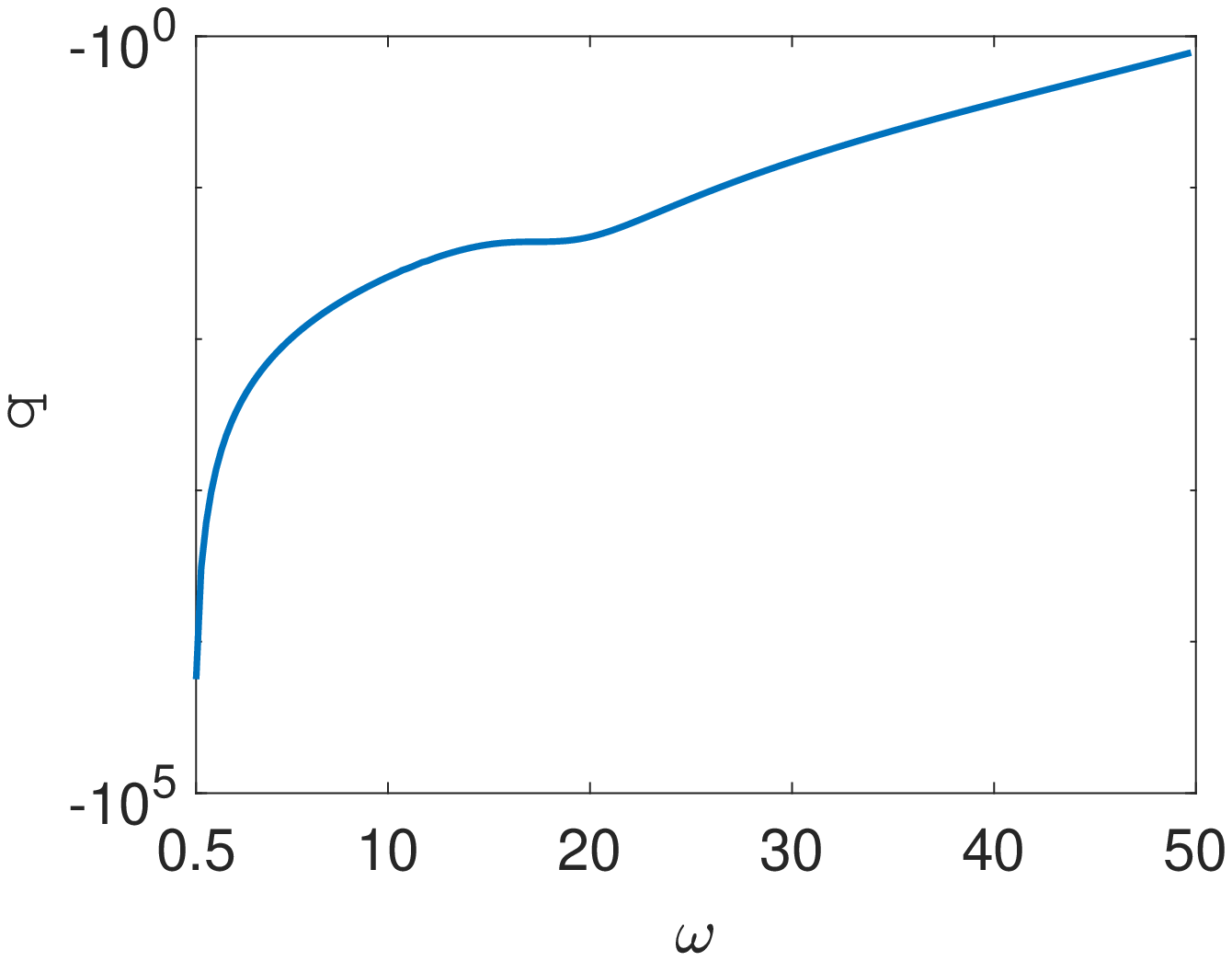}
 \end{minipage}
\hspace{30pt}
\begin{minipage}[t]{0.40\linewidth}
   \includegraphics[width=3.2in]{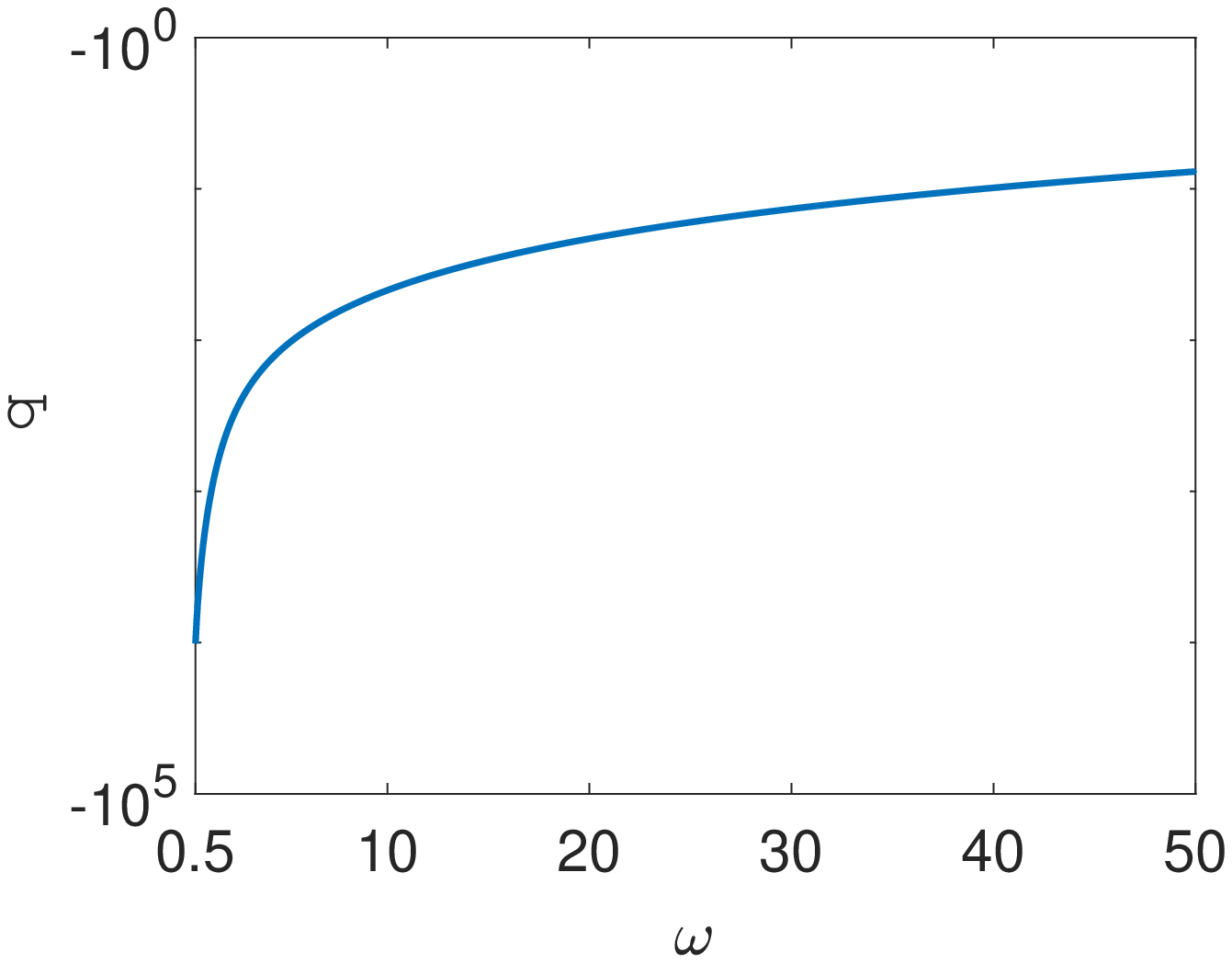}
 \end{minipage}
 \hspace{30pt}
  \begin{minipage}[t]{0.4\linewidth}
  \includegraphics[width=3.2in]{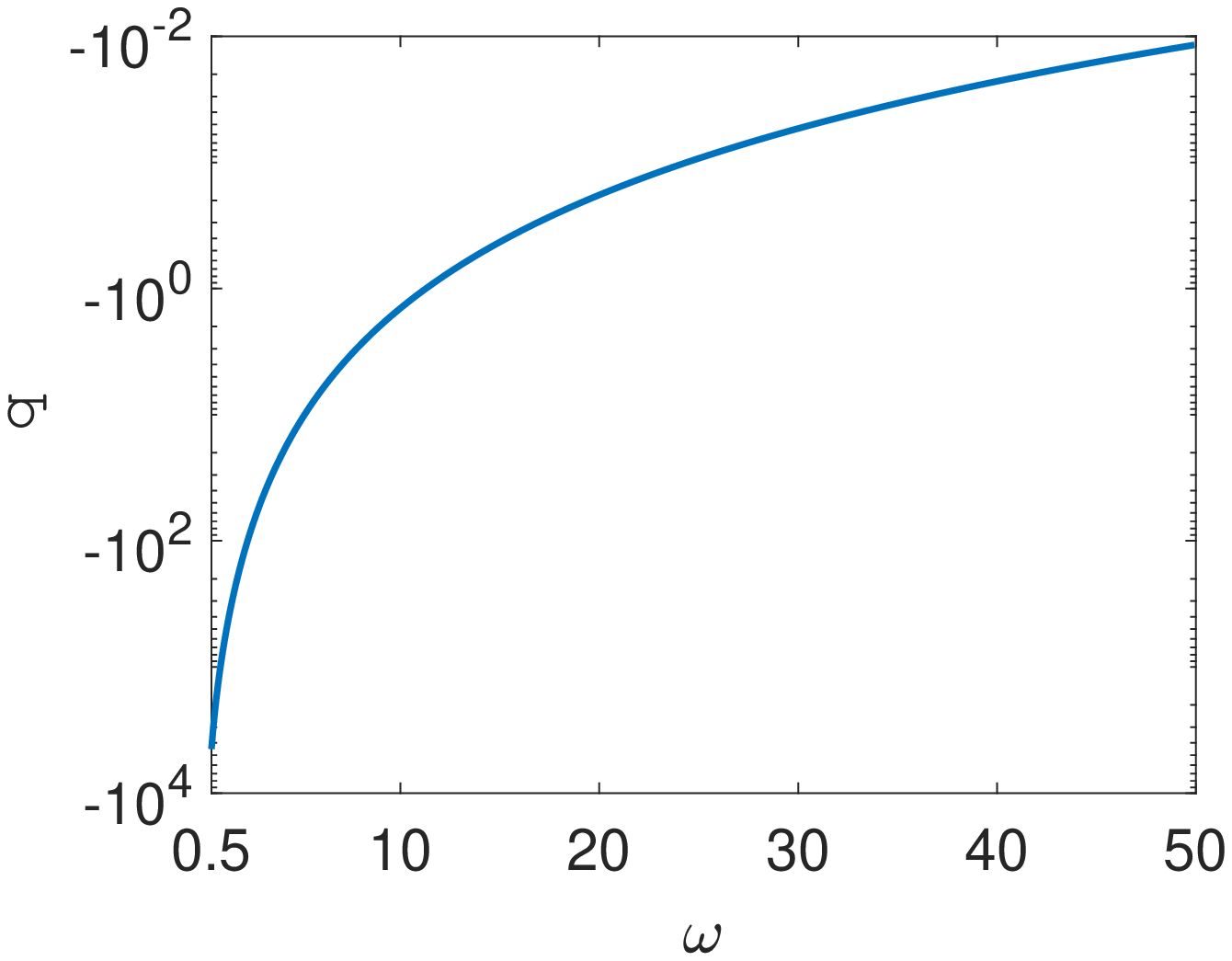}
 \end{minipage}
  \caption{The variation of $\textsf{q}$ (in logarithmic scale) with $\omega$ for $s=0.35$ (top left), $s=0.45$ (top right) and $s=0.5$ (bottom).}
 \label{hessiannegative}
\end{figure}

In order to test the accuracy of our scheme, we compare the exact solution \eqref{dnsol} with the numerical solution obtained by using \eqref{varphi-stokes2} as the initial guess.
The space interval is $ [-\pi,\pi] $ and number of grid points is chosen as $N=2^{10}$. In the first panel of   Figure \ref{dif_er}, we present the exact and numerical solutions for the frequency $\omega=1$. As it is seen from the figure, the exact and the numerical solutions coincide. In the other panels of Figure \ref{dif_er}, the variations of three different errors with the number of iteration are presented.  These results show that our numerical scheme captures the solution remarkably well.

The exact solutions of the fNLS equation are not known for $s\in (0,1)$. In Figure \ref{waves} we illustrate the periodic wave profiles for several values of $s \in (0,1)$ with $\omega=1$. The  nonlinear term becomes dominant with decreasing values of $s\in (0,1)$. Therefore, the wave steepens as expected. In the rest  of the numerical experiments, we investigate the sign of
$\textsf{q}$ in $(\ref{q})$ for different values of $s$.  The interval $\omega \in (1/2, 50]$ is discretized into $1000$ subintervals.
For each value of $\omega$, we generate the periodic wave profile by using the Petviashvili's iteration method on the interval
$ [-\pi,\pi] $ by choosing $N=2^{14}$. We use the forward difference method for the numerical differentiation with respect to $\omega$ after performing the numerical integration.

In Figure \ref{hessiannegative} we illustrate the variation of $\textsf{q}$ with $\omega>\tfrac{1}{2}$ for $s=0.35$, $s=0.45$ and $s=0.5$ where the $\textsf{q}-$axis is in the logarithmic scale. As it is seen from the figure, $\textsf{q}$ is negative. Numerical results indicate that the  periodic wave  is orbitally unstable for $ s \in (0, \tfrac{1}{2} ]$.

\begin{figure}[http]
\hspace{10pt}
 \begin{minipage}[t]{0.4\linewidth}
  \includegraphics[width=3.1in]{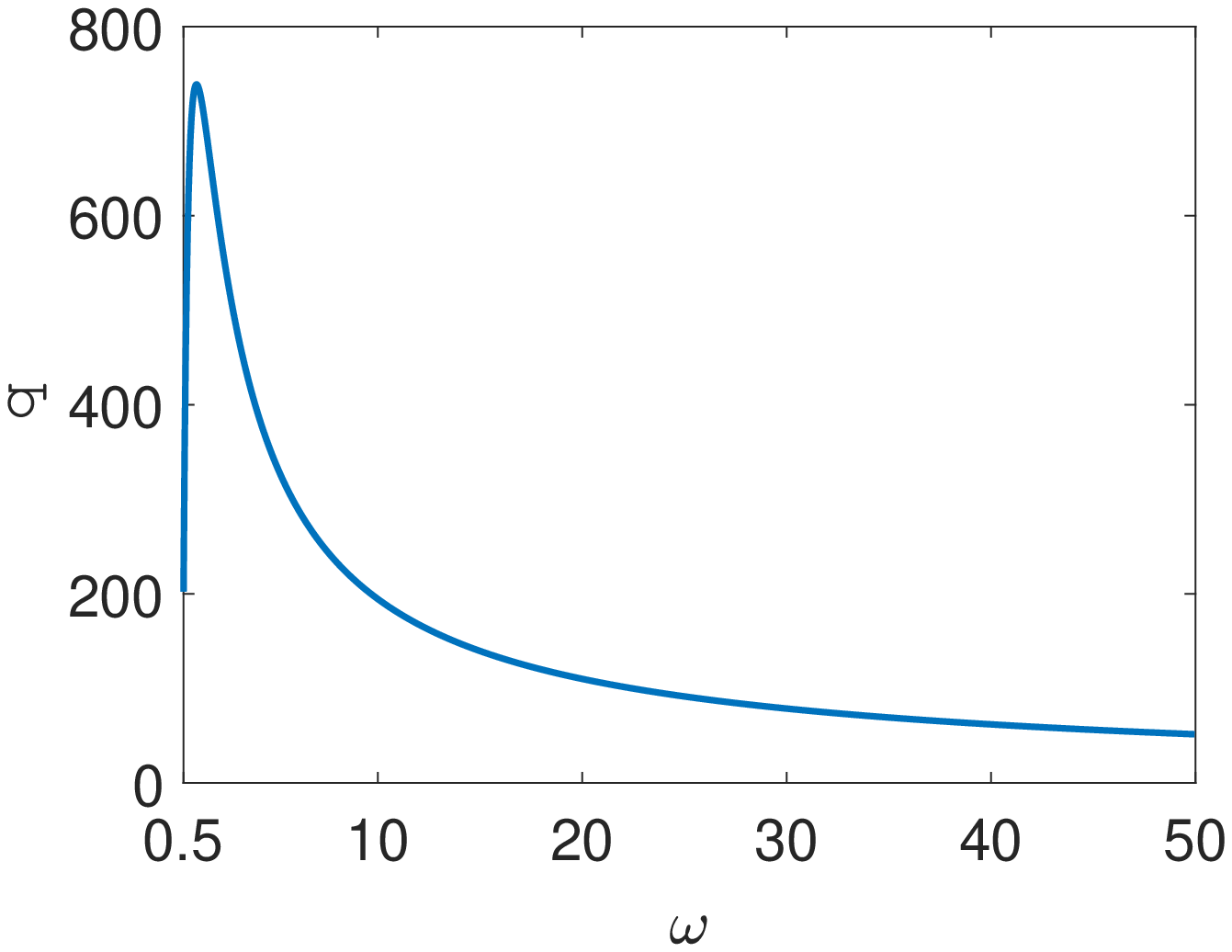}
 \end{minipage}
\hspace{30pt}
\begin{minipage}[t]{0.40\linewidth}
   \includegraphics[width=3.2in]{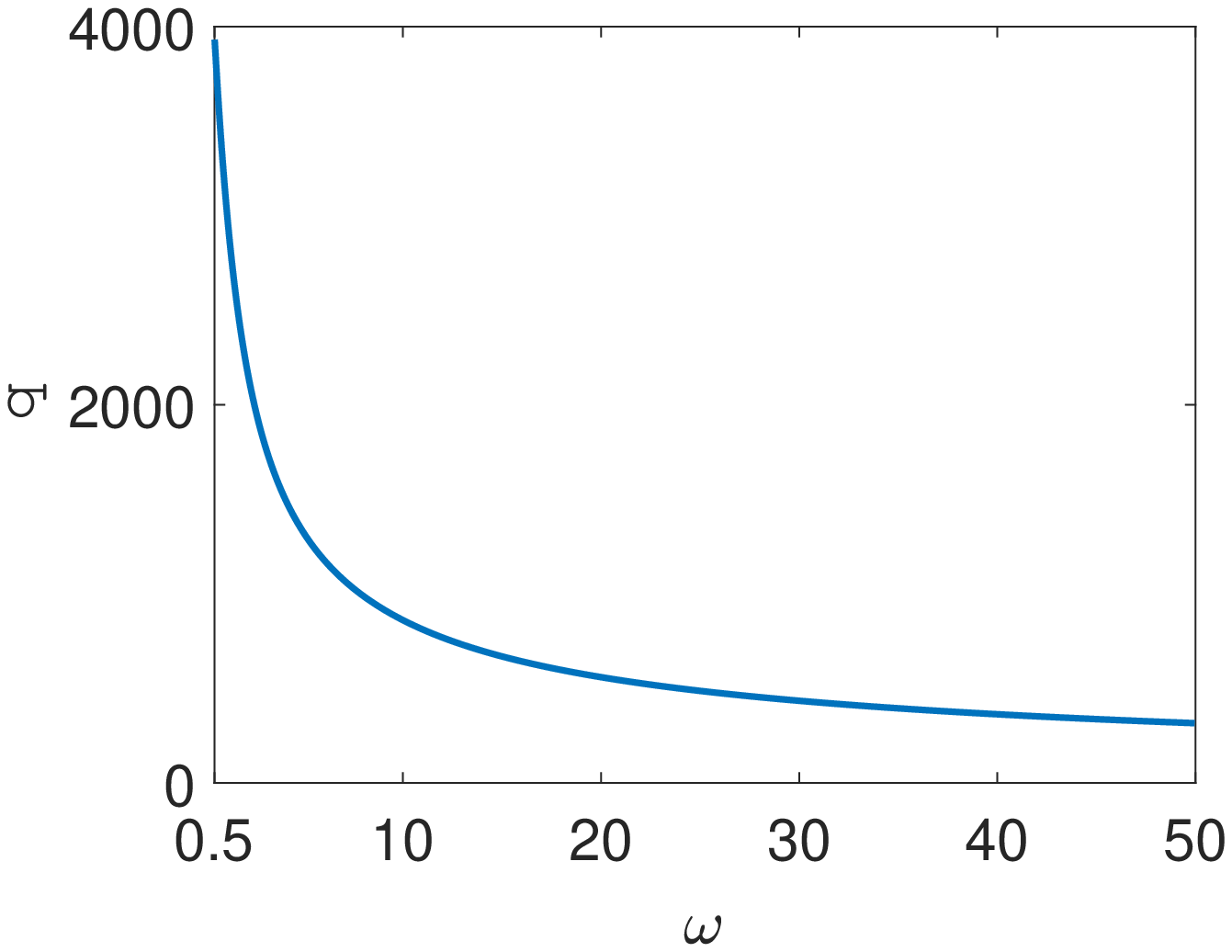}
 \end{minipage}
  \caption{The variation of $\textsf{q}$ with $\omega$ for $s=0.6$ (left panel) and $s=0.8$ (right panel).}
 \label{hessianpositive}
\end{figure}

The variation of $\textsf{q}$ with $\omega>\tfrac{1}{2}$ for $s=0.6$ and $s=0.8$ is depicted in Figure \ref{hessianpositive}.   Numerical results show the  orbital stability of the periodic wave for $s \in \left(\tfrac{1}{2}, 1 \right)$ as  $\textsf{q}$ is positive.

We have performed numerical experiment for several values of $s$. We observe that the value of $\textsf{q}$ is always negative for $s \leq 0.5$  and it  is always positive for $s \geq 0.57$. However, for the values $s\in(0.5, 0.57)$ the numerical results indicate that there is critical wave frequency $\omega_c$ such that $\textsf{q}$ is negative for $\omega<\omega_c$ and positive for $\omega>\omega_c$. In Figure \ref{hessiancritical} the variation of $\textsf{q}$ with $\omega>\tfrac{1}{2}$ for $s=0.52$ and $s=0.55$ is presented.

\begin{figure}[http]
\hspace{10pt}
 \begin{minipage}[t]{0.4\linewidth}
  \includegraphics[width=3.2in]{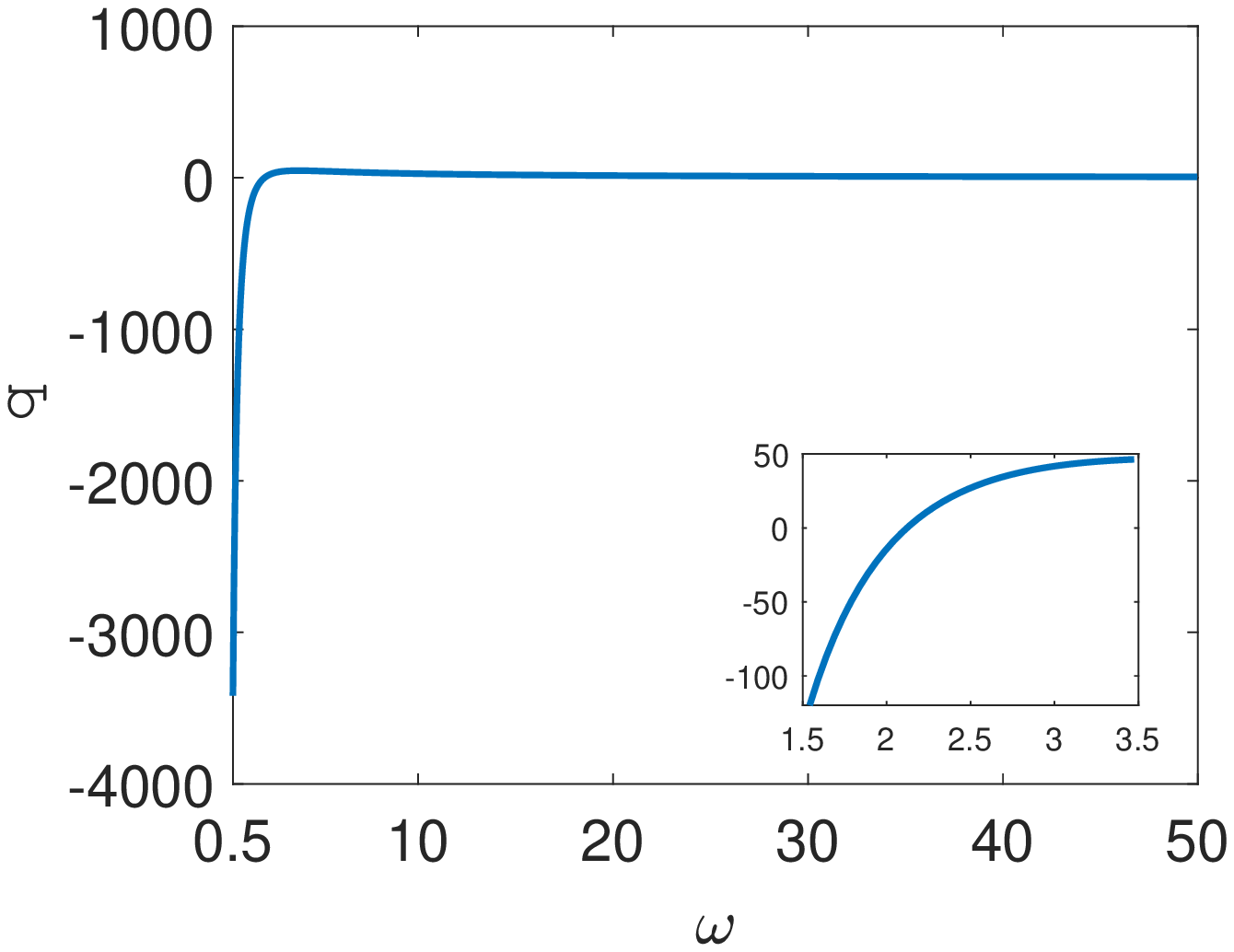}
 \end{minipage}
\hspace{50pt}
\begin{minipage}[t]{0.4\linewidth}
   \includegraphics[width=3.2in]{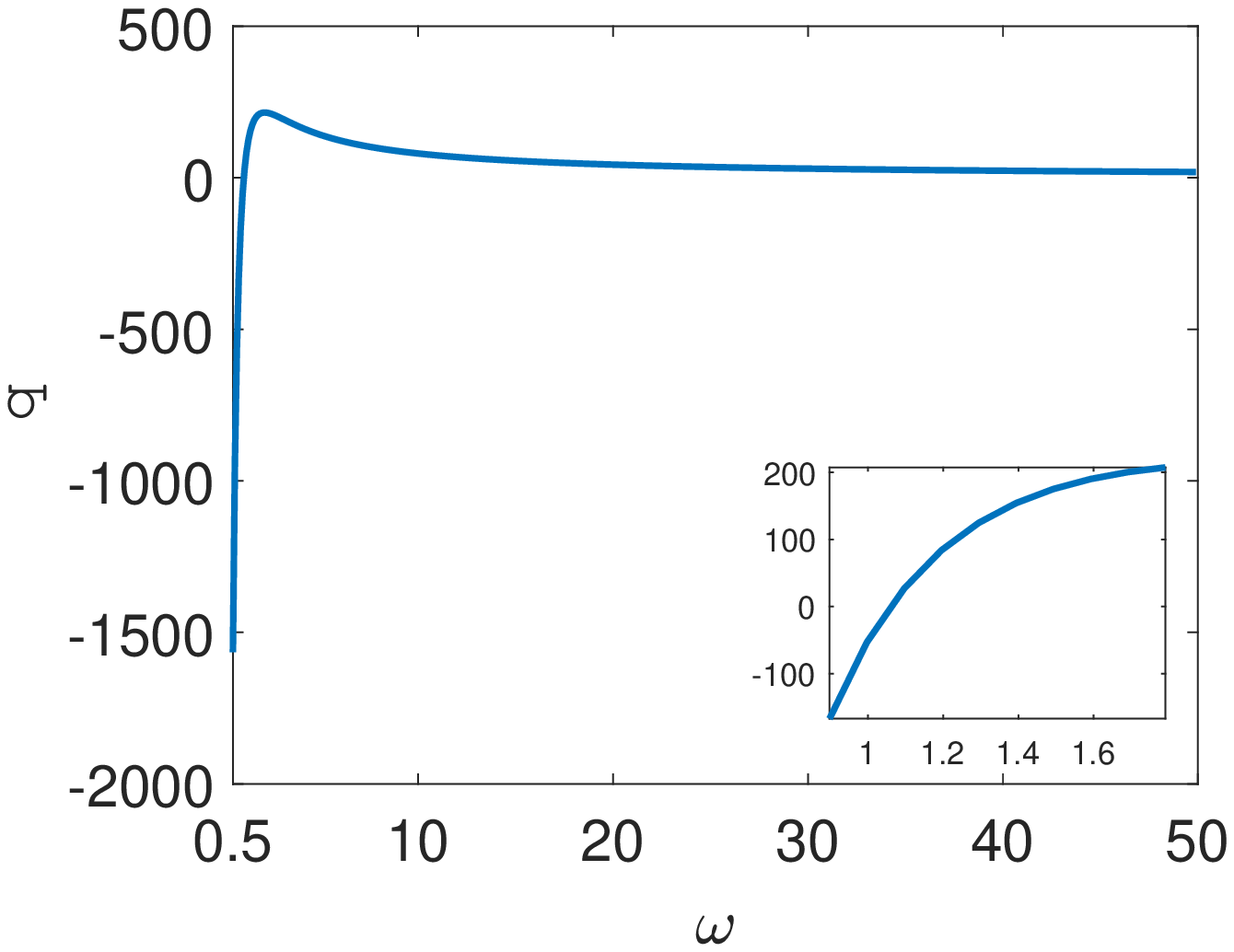}
 \end{minipage}
  \caption{The variation of $\textsf{q}$ with $\omega$ for $s=0.52$ (left panel), and $s=0.55$ (right panel).}
 \label{hessiancritical}
\end{figure}

\newpage

\appendix
\section{Basic facts on positive operators}\label{Angulo-NataliSection}

Here, we  present important facts (verbatim) contained in \cite[Section 2]{Albert} and \cite{AnguloNatali2008}. Let the fixed $\omega>0$  represent the frequency of the periodic wave $\varphi$ which solves equation $(\ref{EDO1})$. For $\theta\geq 0$, we define the operator $S_\theta: \ell^2(\mathbb{Z}) \longrightarrow \ell^2(\mathbb{Z})$ by
$$
S_\theta \alpha(n)=\frac{1}{w_\theta(n)}\sum_{j \in \mathbb{Z}}\mathcal{K}(n-j)\alpha_j=\frac{1}{w_\theta(n)}(\mathcal{K}\ast \alpha)_n,
$$
for all $\alpha=(\alpha_n)_{n\in \mathbb{Z}} \in \ell^2(\mathbb{Z})$, where
$
w_\theta(n)=\beta(n)+\theta+\omega,$ $\mathcal{K}(n)=\widehat{3\varphi}(n),
$
and $\beta(n)$ represents the symbol of the fractional Laplacian given by \eqref{FLaplacian}. Since $w_\theta>0$, we have
$$
X:= \left\{ \alpha=(\alpha_n)_{n \in \mathbb{Z}} \in \ell^2(\mathbb{Z}); \; \|\alpha\|_{X,\theta}:= \left(\sum_{n \in \mathbb{Z}} |\alpha_n|^2w_\theta^2(n)\right)^{\frac{1}{2}}<\infty \right\}
$$
is a Hilbert space with norm $\|\cdot\|_{X,\theta}$ endowed by the inner product
$$
(\alpha_1,\alpha_2)_{X,\theta}= \sum_{n \in \mathbb{Z}} \alpha_{1,n} \overline{\alpha_{2,n}}w_\theta^2(n),
$$
where $\alpha_1=(\alpha_{1,n})_{n\in \mathbb{Z}}$ and $\alpha_2=(\alpha_{2,n})_{n\in \mathbb{Z}} \in \ell^2(\mathbb{Z}).$

\indent By \cite[Proposition 3.2]{AnguloNatali2008}, we see that $S_\theta: X \longrightarrow X$ is well-defined and it defines a compact and self-adjoint operator defined in $X$. By the spectral theorem for compact and self-adjoint operators, there exists an orthonormal basis $(\psi_{i,\theta})_{i \in \mathbb{N}} \subset X$ formed by eigenfunctions of the operator ${S_\theta}_{|_X}$ and corresponding eigenvalues $(\lambda_i(\theta))_{i \in \mathbb{N}} \subset \mathbb{R}$ whose only possible accumulation point is zero. In addition, the eigenvalues can be enumerated of the form
$$
|\lambda_0(\theta)| \geq  |\lambda_1(\theta)| \geq |\lambda_2(\theta)| \geq \cdots .
$$
According to \cite[Corollary 3.1]{AnguloNatali2008}, we have the following important result: $1$ is an eigenvalue of  $S_\theta$  if and only if $-\theta \leq 0$ is an eigenvalue of the operator $\mathcal{L}_1= (-\Delta)^s+\omega  -3\varphi^2$ (as an operator of $L^2_{per}$ with domain $D(\mathcal{L}_1)=H_{per}^{2s}$). Moreover, both eigenvalues have the same multiplicity.

Next, for $\theta \geq 0$   the operator $T_\theta: Y \longrightarrow Y$ is defined by
\begin{equation}\label{Ttheta}
	T_\theta (g):=\left((-\Delta)^s+\theta+\omega\right)^{-1}(3\varphi^2g),
\end{equation}
where
\begin{equation}\label{Y}
	Y:=\left\{ g: \mathbb{R} \longrightarrow \mathbb{C} \; ; \; g \: \text{is $2\pi$-periodic and} \; \|g\|_{Y}:= \left(\frac{1}{2\pi}\int_{-\pi}^{\pi} |g|^2\varphi^2 \; dx\right)^{\frac{1}{2}} < \infty \right\}
\end{equation}
is a Hilbert space with norm $\|\cdot\|_{Y}$ endowed with the inner product
$$
\langle g, h \rangle_Y=\frac{1}{2\pi}\int_{-\pi}^{\pi} g(x)\overline{h(x)} \varphi^2(x)\; dx,
$$
where $g,h \in Y.$
\begin{lemma}\label{prop313-natali}
	Concerning the operator $T_\theta : Y \rightarrow Y$ given by \eqref{Ttheta}, we have
	\begin{itemize}
		\item [(i)] If $g \in Y$ is an eigenfunction of $T_\theta$ for some nonzero eingenvalue, then $g \in {L}^2_{per}$.
		
		\item [(ii)] $T_\theta$ is a compact and self-adjoint operator.
	\end{itemize}
\end{lemma}

\begin{proof}
	\noindent \textit{(i)} Consider $\lambda \neq 0$ satisfying $T_\theta g = \lambda g$. For all $n\in \mathbb{Z}$, we have
	\begin{equation}\label{equalg}
		\widehat{g}(n) = \frac{1}{\lambda} \widehat{T_\theta g}(n) = \frac{1}{\lambda} \frac{1}{w_\theta (n)} \widehat{(3\varphi^2 g)}(n).
	\end{equation}
	
	The fact $g \in Y$ implies that $\varphi^2 g \in {L}^2_{per}$. From the definition of $\omega_\theta$, we obtain by $(\ref{equalg})$ and the Parseval identity (see \cite[Corollary 2.54]{IorioIorio}) that $g \in {L}_{per}^2$, as desired.
	
	\noindent \textit{(ii)} The action of $T_\theta$ over the space $Y$ is given by
	\begin{equation*}
		T_\theta g(x) = \frac{1}{2\pi} \int_{-\pi}^{\pi} G_\theta(x-y) g(y) d\nu (y),
	\end{equation*}
	where $d\nu (y) = \varphi^2(y) dy$ determines a positive measure over the interval $[-\pi,\pi]$ and $G_\theta$ is defined as
	$$\widehat{G_\theta}(n) = \frac{1}{w_\theta(n)}, \;\; n \in  \mathbb{Z}.$$
	Using the Parseval identity we obtain $G_\theta \in {L}^2_{per}$, so that $\tilde{G}_\theta(x,y) = G_\theta(x-y) \in Y\times Y$ and $T_\theta$ is a Hilbert-Schmidt operator. By using \cite[page 264]{Kato}, we see that $T_\theta$ is a compact operator.\\
	\indent On other hand, as $w_\theta$ is an even function, we have that $G_\theta$ is also an even function and, as a consequence, $T_\theta$ is a bounded self-adjoint operator since
	\begin{align*}
		\langle h, T_\theta g \rangle_Y & = \frac{1}{2\pi} \int_{-\pi}^{\pi} h(x) \overline{T_\theta g(x)} \varphi^2(x) dx \\
		& = \frac{1}{2\pi} \int_{-\pi}^{\pi} h(x) \left[ \frac{1}{2\pi} \int_{-\pi}^{\pi} G_\theta(x-y) \overline{g(y)} \varphi^2 dy \right] \varphi^2(x) dx \\
		& = \frac{1}{2\pi} \int_{-\pi}^{\pi} \left[ \frac{1}{2\pi} \int_{-\pi}^{\pi} G_\theta(y-x) h(x) \varphi^2(x) dx \right] \overline{g(y)} \varphi^2(y) dy \\
		& = \langle T_\theta h, g \rangle_Y,
	\end{align*}
	for all $g, h \in Y$.
\end{proof}

By using Lemma $\ref{prop313-natali}$ and the spectral theorem for compact self-adjoint operators, we guarantee the existence of an orthonormal complete set $(\xi_i)_{i \in \mathbb{N}} \subset Y$ formed by eigenfunctions of $T_\theta$ and satisfying
\begin{equation}\label{eigenvalueT}
	T_\theta(\xi_i)=\lambda_i(\theta) \xi_i,
\end{equation}
$\text{for all} \; i \in \mathbb{N}.$

The next result gives us a relationship between the eigenvalues of $T_\theta$ and $S_\theta$.

\begin{lemma}\label{prop314-natali}
	Suppose that ${\rm Ker}(T_\theta) = \{0\}$. Consider $(\xi_i)_{i \in \mathbb{N}}$ an orthonormal complete set of eigenfunctions of $T_\theta$ in $Y$ such that $T_\theta \xi_i = \lambda_i \xi_i$ for all $i \in \mathbb{N}$. Then, $( \sqrt{|\lambda_i(\theta)|} \widehat{\xi_i} )_{i\in \mathbb{N}}$ forms an orthonormal complete set in $X$ constituted by eigenfunctions of $S_\theta$ such that
	$
	S_\theta \widehat{\xi_i} = \lambda_i(\theta) \widehat{\xi_i}
	$
	$\text{for all} \; \; i \in \mathbb{N}.$
\end{lemma}

\begin{proof}
	The condition $\text{Ker}(T_\theta) = \{0\}$ implies that the eigenvalues of $T_\theta$ are non-zero. In addition, by Lemma \ref{prop313-natali} we have $\xi_i \in L^2_{per}$ for all $i \in \mathbb{N}$. By Parseval identity we obtain $\widehat{\xi_i} \in \ell^2 (\mathbb{Z})$ and thus
	\begin{equation*}
		S_\theta \widehat{\xi_i}(n) = \frac{1}{w_\theta(n)} \sum_{i \in \mathbb{N}} \mathcal{K}(n-j) \widehat{\xi_i}(j) = \widehat{T_\theta \xi_i}(n) = \lambda_i \widehat{\xi_i}.
	\end{equation*}
	Then, by \cite[Proposition 3.2]{AnguloNatali2008} we conclude that $\widehat{\xi_i} \in X$ for all $i \in \mathbb{N}$.
	
	We will prove that $\left( \sqrt{|\lambda_i(\theta)|} \widehat{\xi_i} \right)_{i\in \mathbb{N}}$ is an orthonormal complete set. Consider \mbox{$\alpha = (\alpha_n)_{n \in \mathbb{Z}} \in X$}. Using Parseval identity, we have
	\begin{equation*}
		\frac{1}{2\pi} \int_{-\pi}^{\pi} |\widecheck{\alpha}(x)|^2 \varphi^2(x) dx \leq \frac{C}{2\pi} \int_{-\pi}^{\pi} |\widecheck{\alpha}(x)|^2 dx = C \sum_{n=-\infty}^{+\infty} |\alpha_n|^2 < +\infty,
	\end{equation*}
	for some constant $C>0$. Thus, $\alpha \in X$ implies that $\widecheck{\alpha} \in Y$. Using similar arguments, it is possible to show that $g \in Y$ implies $\widehat{g} \in X$.
	
	Next, let $\alpha \in X$ be fixed. By Parseval indentity and using the arguments above, we can conclude that
	$
	\langle \alpha, \widehat{\xi_i} \rangle_{X,\theta} = \lambda_i^{-1} \langle \widecheck{\alpha}, \xi_i \rangle_{Y}.
	$
	Considering $\alpha = \lambda_j \widehat{\xi_j}$, we obtain the basic and useful equality given by
	$
	\langle \lambda_j \widehat{\xi_j}, \widehat{\xi_i} \rangle_{X,\theta} = \frac{\lambda_j}{\lambda_i} \langle \xi_j, \xi_i \rangle_Y.
	$
	Now, suppose that $\alpha \in X$ satisfies $\langle \alpha, \widehat{\xi_i} \rangle_{X, \theta} = 0$ for all $i\in \mathbb{N}$. Thus, $\langle \widecheck{\alpha}, \xi_i \rangle_Y = 0$ for all $i \in \mathbb{N}$. Since $(\xi_i)_{i \in \mathbb{N}}$ is a orthonormal complete set in $Y$, we have that $\widecheck{\alpha} = 0$ in $Y$. On other hand, since $\alpha \in X$ and $Y \hookrightarrow {L}^2_{per}$, we have that $\widecheck{\alpha} = 0$ in ${L}^2_{per}$, that is, $\alpha = 0$ in $\ell^2(\mathbb{Z})$, and therefore $\alpha = 0$ in $X$.
\end{proof}

\begin{remark}\label{remT0}It should be noticed that, if $\theta=0$ then $T_\theta=T_0$ is one-to-one. Indeed, for a given $g \in Y$ such that $T_0(g)=0$, then $\left((-\Delta)^s+\omega\right)^{-1}(3\varphi^2g)=0$ which implies $g=0$. This information is very useful since  we have that $(\xi_i)_{i \in \mathbb{N}} \subset Y$ is an orthonormal complete set in $Y$.
\end{remark}



\section*{Acknowledgments}
G. de Loreno and G. E. B. Moraes are supported by the regular doctorate scholarship from CAPES/Brazil. F. Natali is partially supported by Funda\c c\~ao Arauc\'aria/Brazil (grant 002/2017), CNPq/Brazil (grant 304240/2018-4) and CAPES MathAmSud (grant 88881.520205/2020-01).

\end{document}